\documentclass[12pt]{amsart}
\usepackage{srcltx}
\usepackage{epsfig}
\usepackage{subcaption}
\usepackage{varioref}
\usepackage[                
   pdftex,                  
   pdfstartview=FitH,      
   bookmarks=true,          
	 bookmarksdepth=section,
   bookmarksopen=false,     
   bookmarksnumbered=true
   ]{hyperref}

\allowdisplaybreaks

\DeclareMathOperator{\Aut}{Aut} 
\DeclareMathOperator{\End}{End}
\DeclareMathOperator{\Hom}{Hom}

\DeclareMathOperator{\ev}{ev}
\DeclareMathOperator{\coev}{coev}
\DeclareMathOperator{\tev}{\widetilde{ev}}
\DeclareMathOperator{\tcoev}{\widetilde{coev}}

\newcommand\be{\begin{equation}}
\newcommand\ee{\end{equation}}

\theoremstyle{plain}

\newtheorem{theorem}{Theorem}

\newtheorem{lemma}[theorem]{Lemma}
\newtheorem{proposition}[theorem]{Proposition}
\newtheorem{corollary}[theorem]{Corollary}

\theoremstyle{definition}

\newtheorem{remark}[theorem]{Remark}
\newtheorem{definition}[theorem]{Definition}
\newtheorem{notation}[theorem]{Notations}

\newtheorem{example}[theorem]{Example}

\numberwithin{equation}{section}
\numberwithin{theorem}{section}

\newcounter{ourcount}
\advance\textheight\topskip
\usepackage{bbold}
\usepackage{amsmath,amsfonts,amsthm,amssymb,amscd}
\usepackage{mathtools} 
\allowdisplaybreaks

 \usepackage{graphicx}
 \usepackage[curve,matrix,arrow,color]{xy}

 \usepackage[usenames,dvipsnames]{color}
 \xyoption{line}

\usepackage[mathscr]{eucal}

\newcommand{\id}{\mathrm{id}}

\newcommand{\oN}{\mathbb{N}}

\newcommand{\ok}{{\ensuremath{\Bbbk}}}

\newcommand{\cat}{\mathcal{C}}

\newcommand{\catD}{\mathcal{D}}

\newcommand{\vect}{\mathrm{{\bf vect}}}

\newcommand{\tensor}{\otimes}

\newcommand{\one}{\boldsymbol{1}}%{1\kern-4pt 1}

\newcommand{\Uqua}{U_{\mbox{}\!\!q}\,\mathfrak{sl}_2}

\newcommand{\C}{\ensuremath{\mathbb{C}} }

\newcommand{\CC}{\mathcal{C}}
\newcommand{\ideal}{\mathcal{I}}

\newcommand{\Tr}{\mathrm{Tr}}

\renewcommand{\t}{{\mathsf{t}}}

\newcommand{\tr}{\operatorname{tr}}
\newcommand{\rptr}{\tr^{\Sigma, r}}
\newcommand{\lptr}{\tr^{\Sigma,l}}
\newcommand{\ptr}{\tr^{\Sigma}}
\newcommand{\catd}{\mathcal{C}}
\newcommand{\Id}{\operatorname{Id}}
\newcommand{\Apmod}{A\operatorname{-pmod}}
\newcommand{\Amod}{A\operatorname{-mod}}
\newcommand{\Bpmod}{B\operatorname{-pmod}}
\newcommand{\Bmod}{B\operatorname{-mod}}

\newcommand{\Hxmod}[1]{H_#1\operatorname{-mod}}

\newcommand{\Res}{\mathrm{Res}}

\newcommand{\HH}{\operatorname{HH}_0}

\DeclareMathOperator{\Hpmod}{{\it H}-pmod}
\DeclareMathOperator{\Hmod}{{\it H}-mod}

\newcommand{\coint}{{\boldsymbol{c}}}
\newcommand{\rint}{{\boldsymbol{\mu}}}

\newcommand{\pivot}{{\boldsymbol{g}}}

\newcommand{\ecat}{(\catD,\Sigma)}
\newcommand{\prog}{\Gamma}
\newcommand{\group}{G}
\newcommand{\modcat}{\mathcal{M}}
\newcommand{\modcatN}{\mathcal{N}}
\newcommand{\gp}{\ .}
\newcommand{\gc}{\ ,}
\newcommand{\qcq}{\quad , \quad } % gap-comma-gap 

\newcommand{\ipic}[3][-0.5]{\raisebox{#1\height}{\scalebox{#3}{\includegraphics{pics/#2.pdf}}}}
\newcommand{\ipicc}[3][-0.4]{\raisebox{#1\height}{\scalebox{#3}{\includegraphics{pics/#2.pdf}}}}

\newcommand{\oC}{\mathbb{C}}

\newcommand{\proj}{\mathsf{Proj}}
\newcommand{\tP}{\hat{P}}
\newcommand{\B}{{\mathsf{B}}}
\newcommand{\A}{{\mathsf{A}}}

\begin{document}
\thispagestyle{empty}
\def\thefootnote{\fnsymbol{footnote}}
\begin{flushright}
ZMP-HH/18-17\\
Hamburger Beitr\"age zur Mathematik 747
\end{flushright}
\vskip 4em
%%title
\title{Module traces and Hopf group-coalgebras}

\author[A.F. Fontalvo Orozco]{Andres F. Fontalvo Orozco}  
\address{University of Zurich, I-Math, Winterthurerstrasse 190, CH-8057 Zurich, Switzerland.} 
\email{andres.fontalvoorozco@math.uzh.ch}

\author[A.M.\,Gainutdinov]{Azat M. Gainutdinov}
\address{Institut Denis Poisson, CNRS, Universit\'e de Tours, Universit\'e d'Orl\'eans, Parc de Grammont, 37200 Tours, France. }
\email{azat.gainutdinov@lmpt.univ-tours.fr}

\maketitle
\begin{abstract}
Let $H$ be a finite-dimensional pivotal and unimodular Hopf algebra over a field~$\ok$. It was shown 
 in~\cite{BBG} that the projective tensor ideal in $\Hmod$ admits a unique non-degenerate  modified trace, a natural generalisation of the categorical trace.
This paper provides an extension of this result to a much more general setting.
We first extend the notion of the modified trace to the so-called  module trace
for a given $\ok$-linear module  category $\modcat$ over a pivotal category $\CC$ equipped with a module endofunctor. 
We provide a non-trivial class of examples of such module traces. In particular, we show that any finite-dimensional
pivotal Hopf $\ok$-algebra, not necessarily unimodular, admits a non-degenerate module trace on its projective tensor ideal. We also extend this result to
pivotal Hopf group-coalgebras  of finite type, and we give explicit calculations for the family of Taft Hopf algebras and non-restricted Borel quantum groups at roots of unity.
\end{abstract}

\tableofcontents

\section{Introduction}
Let $\cat$ be a pivotal $\ok$-linear category, which means that $\cat$ is a rigid monoidal category where the left duality functor ${}^*(-)$ agrees monoidally with the right one  $(-)^*$. We will also assume that the endomorphism ring of the tensor unit $\one$  is the ground field $\ok$. This is the context where one can introduce~\cite{Turaev_book} an important notion of the (right) categorical trace $\tr^\cat$. It is a family of $\ok$-linear maps
\begin{align*}
 &\tr^\cat_V\colon \; \End_\cat (V) \to \ok\gc\\
   &\tr_V^\cat(f) = \widetilde{\ev}_V \circ (f\tensor \id_{V^*}) \circ \coev_V \; \in \; \End_\cat(\one)\gc
\end{align*}
where the corresponding endomorphism of $\one$ can be identified with an element of~$\ok$, and here we used the left coevaluation map $\coev_V\colon \one\to V\tensor V^*$ and the right evaluation  map $\widetilde{\ev}_V\colon V\tensor V^*\to \one$.  One can similarly introduce the left categorical trace using the other pair of duality maps.  

The trace $\tr^\cat$ plays a central role in the Reshetikhin--Turaev (RT) construction~\cite{retu, Turaev_book} of isotopy invariants of links, here one should also assume that $\cat$ is ribbon and in this case the left and right categorical traces agree. A coloring of the link components by objects from~$\cat$ provides an endomorphism in $\cat$ that can be evaluated with the categorical trace and this gives a numerical topological  invariant of the link. This invariant can be further extended to the so-called Reshetikhin--Turaev--Witten (RTW) invariant of oriented 3d manifolds via the surgery presentation where one additionally assumes $\cat$ to be semisimple and modular. 

Recent developments in the low-dimensional topology~\cite{Kerler:2001, Geer:2009, TraIdeal, CGP, BCGP, BBGe, RGP} have been paying more attention to \textsl{non-semisimple} ribbon categories.
The belief and a matter of fact is that the non-semisimple  categories might provide finer topological invariants than RTW ones~\cite{CGP,BCGP}. 
The problem here with the RT construction of link invariants  is that it does not basically see a difference between a non-semisimple category~$\cat$ and its semi-simplification $\cat^{\mathrm{s.s.}}$, the reason is that many morphisms in $\cat$ are in the kernel of~$\tr^\cat$, so effectively $\tr^\cat$ factors through $\cat^{\mathrm{s.s.}}$.
 Actually, $\tr^\cat$ provides a non-degenerate pairing of $\mathrm{Hom}_\cat$ spaces if and only if $\cat$ is semisimple. 

The  problem of degeneracy of $\tr^\cat$ appears, for example, in representation categories of quantised Lie algebras $U_q \,\mathfrak{g}$ at $q$ roots of unity and for super Lie algebras even at generic values of $q$.
To overcome the problem of degeneracy of the categorical trace, an axiomatic approach to traces on tensor ideals was proposed in~\cite{Geer:2009, ambi, TraIdeal}. We recall that a \textit{right tensor ideal~$\ideal$} in $\cat$ is a full subcategory closed under retracts (i.e.\ direct summands in the abelian case) and under tensor products of type $X\tensor V$, for $X\in \ideal$ and $V\in\cat$, and similarly for left ideals.

A {\em right  trace on a right tensor ideal $\ideal$}, also called {\em right modified trace}, is  a family of $\ok$-linear functions 
\be\label{def:mod-tr}
\{\t_X\colon \End_\cat(X)\rightarrow \ok \}_{X\in \ideal}
\ee
satisfying cyclicity $\t_X(g \circ f)=\t_{Y}(f  \circ g)$, for any two morphisms $f\colon X\to Y$ and $g\colon Y\to X$, and the  right partial
trace property:
\begin{equation}\label{rpartial}
\t_{X\otimes V}\left(f \right)
=\t_X \left(\,\ipicc{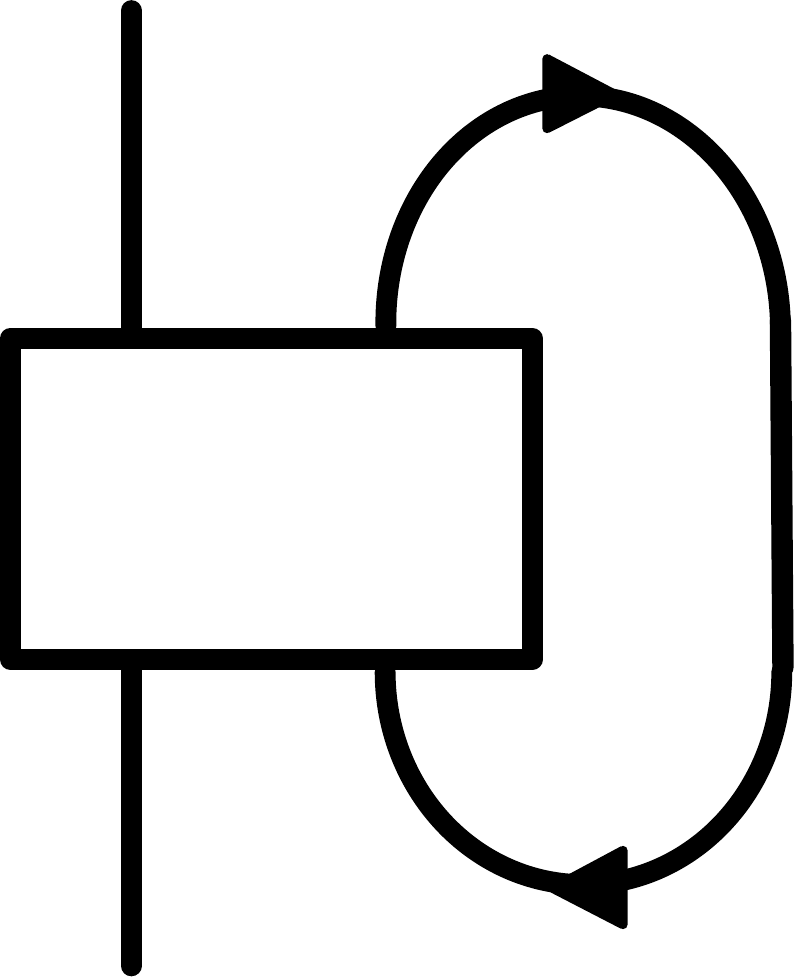}{.16}\, \right)
\put(-42,1){{\footnotesize $f$}}\put(-47,-25){{\tiny $X$}}\put(-38,-17){{\tiny $V$}} 
\end{equation}
for $X\in\ideal$, $V\in\cat$, and any $f\in \End_\cat(X\otimes V)$, and here we used the standard graphical presentation of  $\coev_V$ and  $\widetilde{\ev}_V$ maps. 
A left modified trace on left tensor ideals is defined similarly.
This modified trace comes as a very natural generalisation of the categorical trace, which is the modified trace for the choice $\ideal=\cat$. Indeed, if $\cat$ is semisimple there exists a unique up-to-scalar and non-degenerate solution to the both conditions (cyclicity and partial trace)   proportional to~$\tr^\cat$.
   
 The modified trace plays an important role in several types of generalisations of the RT and RTW constructions.
In particular, an extension based on the Hennings construction~\cite{Hennings} of three-dimensional manifold invariants with links was provided in~\cite{BBGe}. Furthermore, for a factorisable Hopf algebra this construction was extended till the level of 3d TQFTs~\cite{RGP}. 
The starting step here is to construct invariants of framed links with
 components  colored by objects from a tensor ideal $\ideal$ in a given ribbon category $\cat$, e.g.\ the ideal of projective objects. 
 This is like in the RT construction, however an endomorphism in $\ideal$ corresponding to the link should be evaluated with the modified trace $\t$ on $\ideal$ instead of $\tr^\cat$.
As we already mentioned, the categorical trace is very often zero on such endomorphisms and the standard Reshetikhin--Turaev prescription would  give here a trivial invariant.

For a large class of tensor ideals in a pivotal category $\cat$ and under rather technical assumptions on these ideals, modified traces were described in~\cite{ambi}. In particular, it was shown~\cite{GR} that any finite factorisable category (and thus any modular tensor category in the sense of~\cite{Shimizu}) admits a unique and non-degenerate modified trace on its ideal of projective objects. However, an explicit construction was either missing or too complicated for actual calculations.

A good amount of examples of pivotal categories are provided by pivotal Hopf algebras, those where the square of the antipode can be expressed via the conjugation by a group-like element, called \textit{pivot} $\pivot$.
For a finite-dimensional pivotal Hopf algebra $H$ over a field $\ok$
and for its tensor ideal  $\Hpmod$ of projective $H$-modules,   it was recently proven in~\cite{BBG} that a non-zero modified trace  exists and unique (up to a scalar)  under the \textsl{unimodularity} condition on $H$, i.e.\ when the spaces of left and right co-integrals of $H$ agree. Moreover, the trace was constructed using the Hopf algebra integral: the (right) modified trace $\t$ on $\Hpmod$ is uniquely defined by its values on the regular representation $H$, and the formula of~\cite{BBG} says that 
\be
\t_H(f) = \rint\bigl(\pivot \cdot f(1)\bigr)\gc \qquad f\in \End_H(H)\gc
\ee
where $\rint\in H^*$ is the right integral. Typical examples here are the small quantum groups~\cite{Lusztig} associated to any simple Lie algebra $\mathfrak{g}$ of finite rank. In this case, the integrals and thus the modified traces can be computed explicitly, see~\cite[Secs.\,7\,\&\,8]{BBG} for ADE type.

 The result of~\cite{BBG} was then further generalized~\cite{Phu} to the case of pivotal and unimodular Hopf $G$-coalgebras of finite type, where $G$ is a group. Here again, the theory of integrals, now $G$-integrals, plays a crucial role in constructing the modified trace. What is interesting about this last generalisation is that it provides modified traces for a large class of not necessarily finite  categories, here the categories are so-called $G$-graded categories, the only condition is that the category should be grade-wise finite.

Our first extension of the main theorem in~\cite{BBG} is that the modified trace on $\Hpmod$ exists if and only if the algebra is unimodular, the same applies to Hopf $G$-coalgebras. 
Therefore, if a pivotal Hopf algebra is not unimodular its tensor ideal $\Hpmod$ does not posses a non-zero modified trace.  
The famous Hopf algebras such as Sweedler's 4-dimensional one, the Taft Hopf algebras and more generally the Borel parts of quantum groups are all non-unimodular, and therefore they do not have modified traces.  And one of the motivations of our work was to explore further this non-unimodular case and to see if a more general notion of a trace plays a role similar to the modified trace.

In this paper, we first introduce a slight generalisation of the modified trace  for a general module category $\modcat$ over $\cat$ equipped with a module endofunctor $\Sigma$, called below {\em module trace} in Definition~\ref{def:modtrace}. This was inspired by the concept of trace maps on endocategories developed in~\cite{QLHo}. Note that a tensor ideal is nothing but a particular example of a module category with identity endofunctor, we provide also more examples of pairs $(\modcat,\Sigma)$. Such  a generalisation turned out to be a useful framework for studying analogues of modified traces  for Hopf algebras and Hopf $G$-coalgebras which are not necessarily unimodular.

 \newcommand{\modulus}{{\boldsymbol{\alpha}}}

We recall that a Hopf algebra $H$ is non-unimodular if and only if the distinguished group-like element $\modulus\in H^*$, also called \textit{modulus}, is not equal to the counit $\varepsilon$ of $H$. The modulus provides an algebra automorphism:
\be\label{eq:varphi-auto}
\varphi\colon H \to H\gc \qquad h\mapsto (\modulus\tensor \id)\circ\Delta(h)
\ee
and therefore this defines an endofunctor $\varphi_*$ on $\Hpmod$ given by the twisting $P\mapsto {}_\varphi P$,  where  ${}_\varphi P$ is the same vector space as $P$ but with the $H$-action twisted by $\varphi$. Note that this functor is the identity if and only if the algebra is unimodular. A first reasonable problem here is  existence  of a module trace on the module category $\modcat=\Hpmod$ equipped with the module endofunctor $\Sigma=\varphi_*$. In other words, we are interested in a certain ``twisted" version of the modified trace on $\Hpmod$ which is given by
 a family of $\ok$-linear maps 
  \be
  \{\t_P:\Hom_\cat\bigl(P,\varphi_*(P)\bigr)\rightarrow\ok\}_{P\in\Hpmod}
  \ee
which are $\varphi_*$-cyclic in the sense 
 \be
 \t_P(f\circ g)=\t_{P'}\big(\varphi_* (g)\circ f\big)
 \ee
  for all   morphisms $P\stackrel{g}{\rightarrow}P'$ and $P'\stackrel{f}{\rightarrow}\varphi_*(P)$,
  and the maps $\t_P$  satisfy a twisted analogue of the right partial trace condition formulated in~\eqref{eq:rpartial}.
Our main result is that any finite-dimensional pivotal Hopf algebra admits such a
module trace, and it is again, as in~\cite{BBG}, related to the space of integrals in $H$.

\begin{theorem}
Let $(H,\pivot)$ be a finite dimensional pivotal Hopf algebra over a field $\ok$ with modulus $\modulus$. The space of right module traces on $(\Hpmod,\varphi_*)$ is $1$-dimensional
and the traces are determined uniquely  by its values on the regular representation:
\be
 \t_H(f)=\rint\bigl(\pivot\cdot f(1)\bigr)\gc \qquad f\in \Hom_H(H,{}_\varphi H)\gc
\ee 
where $\rint$ is the right integral.

In particular, $\Hpmod$ has a non-zero modified trace if and only if $H$ is unimodular. 
\end{theorem}

This theorem is a part of Corollary~\ref{cor:Hopf} where we prove a more general result: if one replaces the modulus $\modulus$ in~\eqref{eq:varphi-auto} by any other group-like functional on $H$ and considers the corresponding module endofunctor, then the space of module traces is zero dimensional.

Moreover, we actually work in the $G$-graded framework and most of the statements are made for Hopf $G$-coalgebras. In particular,  the full version of the above theorem (including the case of left traces) is Theorem~\ref{thm:main}. The ordinary Hopf algebra case can be always recovered choosing $G$ to be the trivial group. We provide also examples of explicit calculations of the module trace for the already mentioned non-unimodular Hopf algebras: for Taft algebras, and in the case of non-trivial  Hopf $G$-coalgebras, we analyse the positive Borel part of the unrestricted quantum $sl(2)$ at roots of unity. Here, $G$ is some non-finite non-abelian group.

For  module traces in general, we also provide a pull-back construction.  A sufficiently nice functor between module categories enables us to transport trace maps from the target category.  This gives our second  result: an explicit construction of large classes of tensor ideals in $\Hmod$ for $H$ a Hopf group-coalgebra, and module traces on them. The details are in Section~\ref{sec:pullback}. We  consider an application of this pull-back construction to a large class of Hopf algebras related to small quantum groups in our second paper~\cite{FG}.

We  believe that our theory of twisted traces (or module traces more generally) should provide an important ingredient for an extension of the Kuperberg invariants~\cite{Ku} for manifolds involving links, or more generally ribbon graphs. Recall that in the Kuperberg construction a special case is played by so-called balanced Hopf algebras where the square of the antipode is expressed via a conjugation with the distinguished group-like element, we called it here \textit{co-modulus}. Typical examples of such Hopf algebras are provided by non-unimodular Hopf algebras like Borel parts of small quantum groups. Therefore, it would be interesting to study what role our module trace plays in the Kuperberg construction.

We should  mention that while working on this paper we  learned that
 N. Geer, J. Kujawa, and B. Patureau-Mirand
were defining~\cite{GKP18} a notion ``m-trace" which is related to our module trace. However, they use  different techniques and their work generalises the theory of modified traces  as established
in \cite{ambi} in the unimodular case.

The paper is organised as follows. We begin with description of some known results on trace maps on endocategories in Section~\ref{sec:endoc}. In this section we also describe our conventions on module categories and give a definition to module traces, with several examples, and establish few important results on them, like the pull-back construction and the so-called Reduction lemma. Then, we briefly discuss $\group$-categories and their module categories in Section~\ref{sec:Gcat}. In Section~\ref{sec:HopfGC}, we make a short survey of Hopf $\group$-coalgebras and introduce a family of endofunctors on their categories of modules. In Section~\ref{sec:main}, we state and prove our main theorem. In Sections~\ref{sec:taft} and~\ref{sec:Borel}, we consider an application of the main theorem in explicit calculations of module traces for the Taft Hopf algebra and the Borel part of the non-restricted quantum $sl(2)$ at roots of unity. In the appendix, we recall basic definitions around module categories.

\subsection*{Acknowledgements} The authors are grateful 
to NCCR SwissMap for generous support
and to Anna Beliakova, Christian Blanchet, Alexei Davydov, Ngoc Phu Ha, Krzysztof Karol Putyra, 
and Ingo Runkel
for helpful discussions. 
The authors are also thankful to the organizers of conference ``TQFT and Categorification" in Cargese in April, 2018, where some of our results were presented. AFFO thanks the University of Tours for their hospitality during June 2018. AFFO is supported by NCCR SwissMap.  AMG also thanks Institute of Mathematics in Zurich University for kind hospitality during 2017 and 2018.
AMG is supported by CNRS and also thanks the Humboldt Foundation for a partial financial support. 

\section{Endocategories and traces}\label{sec:endoc}

Given a category $\catD$ together with an endofunctor $\Sigma\colon\catd\rightarrow\catd$, we call the pair $\ecat$ an endocategory. Any category can be seen as an endocategory with the identity endofunctor $\Sigma=\Id$. In this case we often write $\catD$ instead of $(\catD,\Id)$. In case $\catD$ carries extra structure, it is assumed, unless otherwise specified, that $\Sigma$ respects this structure in an appropriate sense, for instance, we say $\ecat$ is $\ok$-linear if both $\catD$ and $\Sigma$ are.
For $X\in\catD$ and a morphism $g$, we will  use the short-hands $\Sigma X$  and $\Sigma g$ for $\Sigma(X)$ and  $\Sigma(g)$, respectively. The morphism spaces will be denoted by $\catD(X,Y)$.

The following definition is due to~\cite{QLHo}. 
\begin{definition}\label{def:endotr}
 Given a $\ok$-linear endocategory $\ecat$, a \textit{trace map} on it  is a family of $\ok$-linear maps 
 $$
 \{\t_X
 \colon
 \catD(X,\Sigma X)\rightarrow \ok\}_{X\in\catD}
 $$ 
 such that 
 \be\label{eq:cyc}
 \t_X(f\circ g)=\t_Y\big(\Sigma (g)\circ f\big)
 \ee
  for all   morphisms $X\stackrel{g}{\rightarrow}Y$ and $Y\stackrel{f}{\rightarrow}\Sigma X$.
   We also say in this case that the family $\t$ is \textit{$\Sigma$-cyclic}.
 \end{definition}

 \begin{remark}
 Taking in the above definition $Y=\Sigma X$ and $f=\id$, we see  that trace maps on an endocategory $\ecat$ are $\Sigma$-invariant: 
 $$
 \t_X(g)=\t_{\Sigma X}(\Sigma g)\gp
 $$
  Furthermore, the  case when $\Sigma$ is the identity functor recovers the usual notion of a trace map on a $\ok$-linear category. In particular, the categorical trace $\tr^\CC$ on a pivotal category $\CC$ is a trace map in this sense. 
 \end{remark} 
  A useful construction on the study of trace maps on endocategories is the so-called Hochschild-Mitchel complex $\mathcal{CH}_\bullet\ecat$, first studied in the untwisted case, $\Sigma=\Id$, in \cite{Mitch}, and in the general case in \cite[Sec.\,4]{QLHo}. It is a simplicial vector space defined as follows
 \be\label{def:HMcomplex}
   \mathcal{CH}_n\ecat:=\bigoplus_{X_0,..,X_n\in\catD}\catD(X_0,\Sigma X_n)\otimes_\ok\catD(X_1,X_0)\otimes_\ok...\otimes_\ok\catD(X_n,X_{n-1})\\
\ee
  with face maps
\be \label{def:CHface}
  d_i(f_0\otimes_\ok...\otimes_\ok f_n):=\left\lbrace\begin{array}{lc}f_0\otimes_\ok...\otimes_\ok f_i\circ f_{i+1}\otimes_\ok...\otimes_\ok f_n & i\neq n \\ \Sigma f_n\circ f_0\otimes_\ok f_1\otimes_\ok...\otimes_\ok f_{n-1}&i=n\end{array}\right.
\ee
and degeneracies
\be
 s_i(f_0\otimes_\ok...\otimes_\ok f_n):= f_0\otimes_\ok...\otimes_\ok f_i\otimes_\ok \id_{x_i}\otimes_\ok f_{i+1}\otimes_\ok...\otimes_\ok f_n\,.
\ee 
 
As usual we can see it as a chain complex with boundary maps $\partial=\sum_i(-1)^id_i$. The homology of this complex is denoted by $\operatorname{HH}_\bullet\ecat$ and its $0$-th degree is given by
\be\label{endotra}
\HH\ecat=\bigoplus_{X\in\catD}\catD(X,\Sigma X)\Big/\langle f\circ g-\Sigma(g)\circ f \rangle.
\ee
It is easy to see that trace maps on $\ecat$ are in bijective correspondance with linear forms on $\HH\ecat$. In other words, this space is the universal trace $\Tr\ecat$ of \cite[Sec.\,3]{QLHo}.

\subsection{Module categories and conventions}
In what follows, we assume $\CC$ to be a $\ok$-linear monoidal category.
A brief introduction to $\CC$-module categories and $\CC$-module functors can be found in Appendix \ref{app:modcat}, the reader is also encouraged to check with \cite[Ch.\,7]{tensor} for a more complete account of the module category theory. 

Let $\modcat$  be a right $\CC$-module category and $\Sigma\colon \modcat\rightarrow\modcat$ a $\CC$-module endofunctor  with 
its  module structure
 \begin{align}\label{eq:mod-str}
\sigma_{P,V}&\colon\Sigma(P\odot V)\rightarrow \Sigma(P)\odot V\gp
\end{align}
  
 First, we introduce relevant graphical notations for module categories.  We use standard string diagrams for morphisms, e.g.\ a morphism $f\colon P\rightarrow P'$ in $\modcat$ will be represented by the diagram with a  coupon:
 $$
  \ipic{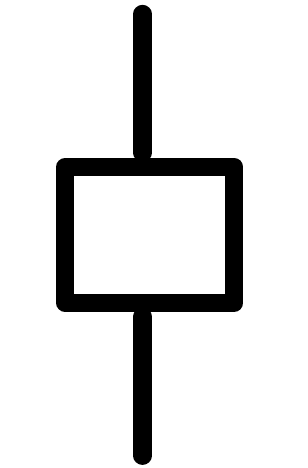}{0.3}
 \put(-17,-2){\tiny $f$}\put(-16,-28){\tiny$P $}\put(-16,22){\tiny$P'$}\ ,
 $$
 and we read the diagrams from bottom to top.
 The action of an endofunctor $\Sigma$ will be represented by dashed lines to the left and the right of the map  it is applied to. For example,
 we will write
  \be
  \ipic{f}{0.3}
  \put(-20,-2){\tiny $\Sigma f$}
  \put(-20,-28){\tiny $\Sigma P$}
  \put(-20,22){\tiny $\Sigma P' $}
  \;=\;\ipic{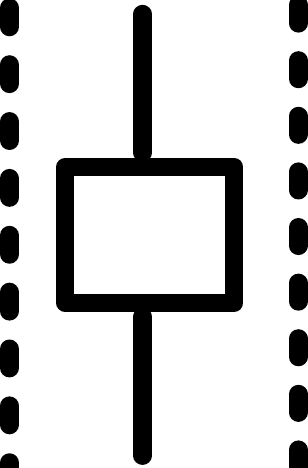}{0.3}\put(-17,-3){\tiny $f$}\put(-18,-28){\tiny$P $}\put(-18,20){\tiny$P'$}\ .
  \ee
 Furthermore,  the (right) module structure~\eqref{eq:mod-str}
   will be represented with the diagram
 \be
 \ipic{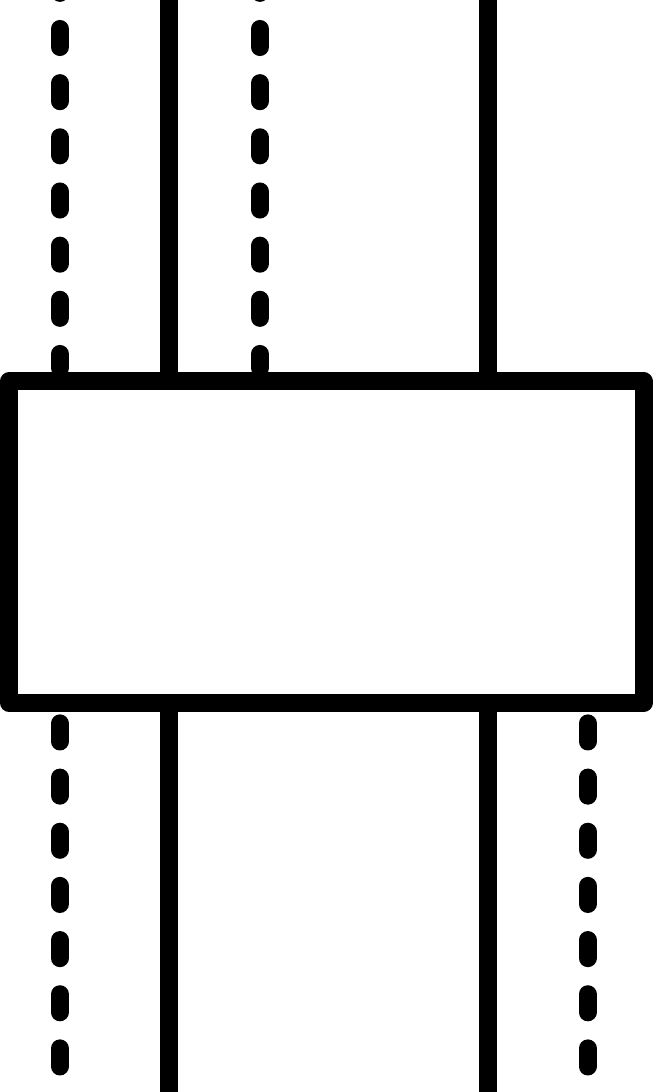}{0.2}\put(-25,0){\scriptsize $\sigma_{P,V}$}\put(-30,-39){\tiny$P$}\put(-13,-39){\tiny$V$}
 \; =\; 
 \ipic{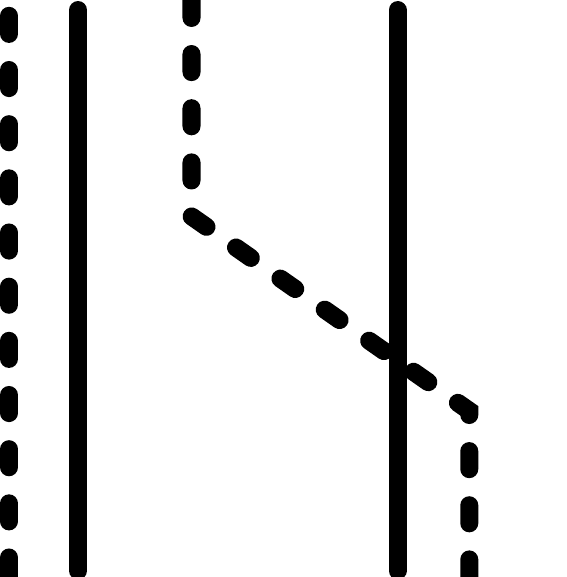}{0.3}\put(-45,-35){\tiny$P$}\put(-18,-35){\tiny $V$}
 \ee
and its inverse with the diagram
\be
\ipic{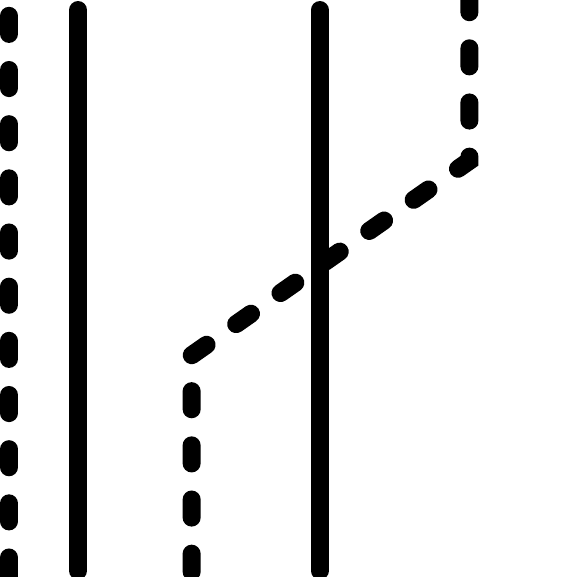}{0.3}.
\ee

 The naturality of $\sigma$ can then be interpreted as the ability of morphisms in $\CC$ to ``move through dashed lines on the right",
 or diagrammatically,
 \be\label{diag:natstr}
 \ipic{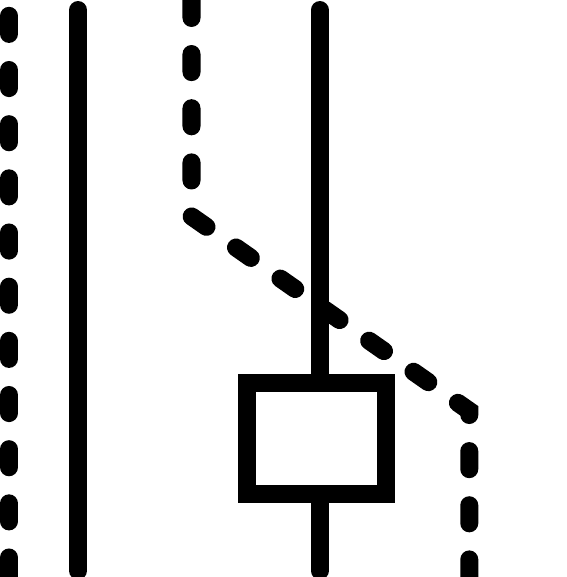}{0.3}
\put(-24,-16){\tiny $f$} 
 =
 \ipic{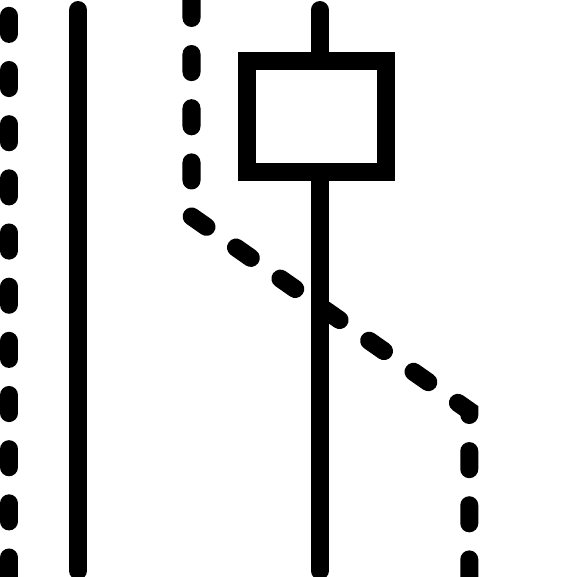}{0.3}\put(-24,12){\tiny $f$}.
 \ee
 
\newcommand{\otact}{\odot}

\subsection{Examples}\label{sec:examples} Here, we recall the following list of rather abstract examples of module categories

\begin{enumerate}
\item A tensor category $\CC$ is a module category over itself, both right and left, with the action $X\otact Y := X\tensor Y$. For $X\in\CC$, the left tensoring functor, $L_X=X\otimes -$, is a right module functor with $\sigma_{Y,Z}=\alpha_{X,Y,Z}$.

\item  A left/right tensor ideal $\ideal$ of $\CC$ can be regarded as a left/right $\CC$-module with a module functor $\Sigma=\Id$.  In particular, the subcategory $\proj(\cat)$ of projective objects is an important example of a tensor ideal. If $\CC$ is pivotal, then $\CC$ has   the categorical trace $\tr^\CC$, and the kernel of $\tr^\CC$ is an interesting instance of a tensor ideal too. However these ideals are non-abelian as module categories if $\CC$ is non semi-simple.

\item Any additive category $\CC$ has a $\vect_\ok$-module structure: for $V\in \vect_\ok$ and $X\in \CC$ we have the action $V\otact X := X^{\oplus \dim(V)}$.

\item Having an algebra object $A \in \CC$, we can consider the category  $\Amod$ of left $A$-modules in $\CC$, i.e.\ pairs $(M,\rho_M)$ where $M\in\CC$ and a morphism $\rho_M\colon A\tensor M\to M$ satisfying obvious ``module map" conditions. $\Amod$ is a right $\CC$-module category by the action $M\odot X := M\tensor  X$ where $M\in \Amod$ and so $M\tensor X$ is clearly a left $A$-module with the action morphism $\rho_M\tensor \id_X$.
We note that $\Amod$ is abelian, and actually any (exact) module category can be realized this way:
for a given exact module category $\modcat$, its algebra $A$ with $\Amod\cong \modcat$ can be reconstructed as an internal Hom, see more details in~\cite{tensor}.

\item A factorizable tensor category $\CC$ has naturally a module structure over its Drinfeld centre $Z(\CC)\cong \CC\boxtimes \CC^{\mathrm{op}}$ with the action $(X\boxtimes Y) \otact M := X\tensor M\tensor Y$.
\end{enumerate}

Hopf algebras provide another class of examples of module categories and their endofunctors:

\begin{enumerate}
\item A Hopf algebra $H$ with an algebra automorphism $\varphi\colon H\to H$ satisfying 
$$
\Delta\circ\varphi=(\id\otimes \varphi)\circ\Delta
$$
 makes $\Hpmod$ a module category over $\Hmod$ with the module endofunctor given by twisting the action of $H$ with $\varphi$. This will be discussed more below.

\item A (right) co-ideal subalgebra $A$, in particular a Hopf subalgebra, in $H$ provides a tensor ideal in $\Hmod$ under the ``pull-back" construction: define $\ideal_A$ to be the full subcategory of $\Hmod$ whose objects are projective as $A$-modules. In other words,
\be\label{eq:IA-def}
\ideal_A := \{\,V\in \Hmod \; | \; \mathrm{Res}(V)\in \Apmod\, \}\gc
\ee
where $\mathrm{Res}\colon \Hmod \to \Amod$ is the restriction functor.
Note that $\Apmod$ is a right module category over $\Hmod$, and $\mathrm{Res}$ is a right module functor. It then follows that $\ideal_A$ is  a right tensor ideal in $\Hmod$, in particular   $I_H$ is $\Hpmod$ and $I_{\langle 1\rangle}=\Hmod$. We provide more details and proofs in Section~\ref{sec:pullback}.
\end{enumerate}

\subsection{Module trace}
For introducing module traces,
we now assume that $\CC$ is a 
pivotal category. 
Recall that in a pivotal category left and right duals agree. Hence, besides usual evaluation and coevaluation morphisms
 $$
 {\ev}_V\colon V^*\otimes V\rightarrow \mathbb{1}\ , \qquad {\coev}\colon \mathbb{1}\rightarrow V\otimes V^*\ ,
 $$
satisfying the zig-zag identities,  there are morphisms 
 $$
 \widetilde{\ev}_V\colon V\otimes V^*\rightarrow \mathbb{1}\ , \qquad \widetilde{\coev}\colon\mathbb{1}\rightarrow V^*\otimes V
 $$
  called right evaluation and right coevaluation respectively, also subject to  the zig-zag identities.
 We follow graphical conventions for these maps given in~\cite{BBG}.

 \begin{definition}
Let $\modcat$ be a right module category over $\CC$ and $\Sigma$ a right $\CC$-module endofunctor on $\modcat$. 
 For $P\in\modcat$, $V\in\catd$ and $f\colon P\otact V\rightarrow \Sigma(P\otact V)$, 
 the $\ok$-linear map
\begin{align} \label{E:PartialLRtrace}
\rptr_V(f)&:=(\id_{\Sigma P} \otact \tev_V)\circ\bigl((\sigma_{P,V}\circ f )\otact \id_{V^*}\bigr)\circ(\id_P \otact \coev_V) \\
& = \; \put(15,-18){{\tiny $V$}}
\;\ipicc{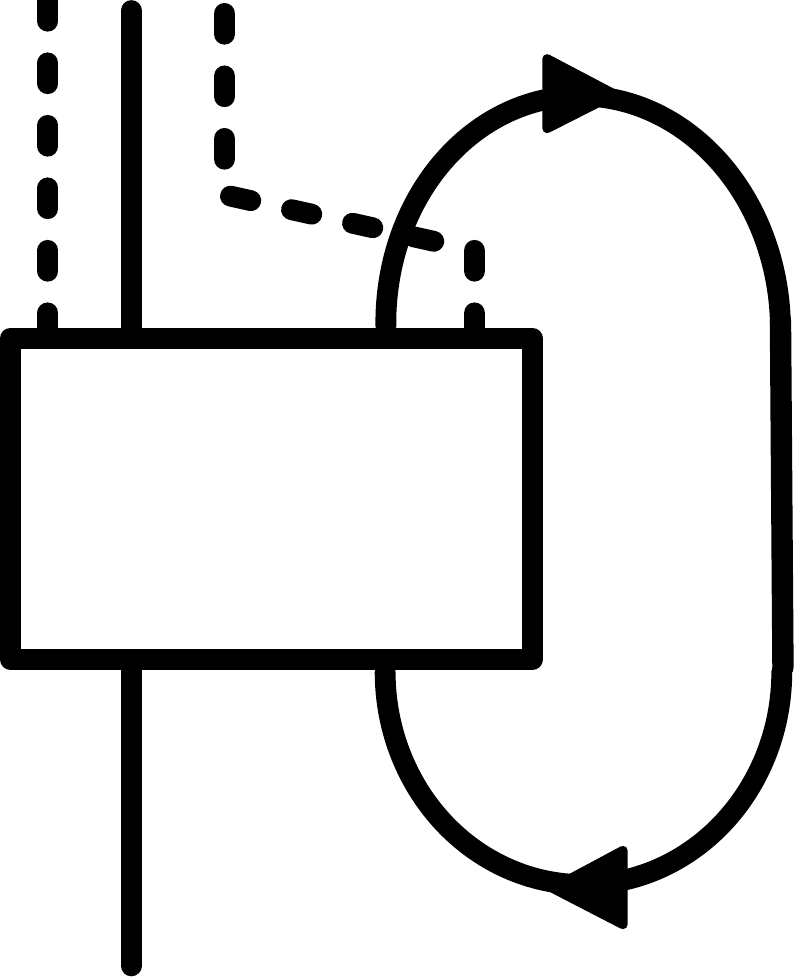}{.16} \put(-27,1){{\footnotesize $f$}}\put(-34,-25){{\tiny $P$}} \;\in \;\modcat(P,\Sigma P) \notag
\end{align} 
is called  \textit{right partial $\Sigma$-trace}.
\end{definition}

 \begin{definition}\label{def:modtrace}
  A \textit{right module trace} on a right $\CC$-module endocategory $(\modcat,\Sigma)$ is a family of $\ok$-linear maps 
  $$\{\t_P\colon\modcat(P,\Sigma P)\rightarrow\ok\}_{P\in\modcat}$$
   subject to the following two conditions:
 \begin{description}
 \item [$\Sigma$-cyclicity] $\t$ is a trace map on $(\modcat,\Sigma)$ in the sense of Definition~\ref{def:endotr}, 

\item[Right partial trace property]  for $P\in\modcat$, $V\in\CC$ and $f\colon P\otact V\rightarrow \Sigma(P\otact V)$
\be \label{eq:rpartial}
\t_{P\otact V}(f)=\t_P\bigl(\rptr_V(f)\bigr)\ .
\ee
 \end{description}
 Furthermore, we call a trace map $\t$ {\em non-degenerate} if the pairings
 \be
  \modcat(P',\Sigma P)\times\modcat(P,P')\rightarrow\ok\qcq (f,g)\mapsto \t_P(f\circ g)
 \ee
  are non-degenerate for every $P,P'\in\modcat$.
 \end{definition}

 \begin{remark}
 \mbox{}
 \begin{enumerate}
 \item
One has similarly a definition of a left module trace for $\modcat$ a left $\CC$-module category with a left module endofunctor $\Sigma$. Here, we define the left partial $\Sigma$-trace  by the map
 $$
 \lptr_V(f):=(\ev_V \otact\,\id_{\Sigma P})\circ\bigl(\id_{V^*}\otact(\sigma_{V,P}\circ f) \bigr)\circ(\tcoev_V \otact\,\id_P)\ ,
 $$
  where $\sigma_{V,P}\colon\Sigma(V\otact P)\stackrel{\cong}{\rightarrow} V\otact\Sigma(P)$. We say $\t$ is a left module trace if it is $\Sigma$-cyclic and satisfies a left partial trace property analogous to \eqref{eq:rpartial}.
  
 \item
  The theories of left and right module traces are essentially the same: one is the other in the category with the opposite monoidal structure and hence we study in details only one of them, the right one. For convenience, in what follows we will also write $\ptr=\rptr$.
   
   \item Recall from examples in Section~\ref{sec:examples}, that 
 a tensor ideal $\ideal$ of $\CC$ provides a $\CC$-module with the identity module functor. The \textit{modified trace} introduced 
in~\cite{ambi, TraIdeal} can then be recovered as a special case of the above definition of the module trace for $(\ideal,\Id)$. 
 \end{enumerate}
  \end{remark}

It is of course an interesting problem to classify module traces for the class of examples  from Section~\ref{sec:examples}. In the case $\CC$ is a pivotal category, then the module trace on the module endocategory $(\CC,\Id)$ is unique up to a scalar and is given by the categorical trace $\tr^{\CC}$. However, this trace is non-degenerate if and only if $\CC$ is semisimple. For tensor ideals in $\CC$, a classification of module traces (or modified traces in this case) is only partially known, see \cite{BBG,genTrace,TraIdeal,GR}. 
 More explicit examples of module traces are presented below in the case of (non-unimodular) Hopf $\group$-coalgebras.

 \newcommand{\fun}{\mathcal{F}}
 \subsection{Pull-back of module traces}\label{sec:pullback}
 Let $\modcat$ and $\modcatN$ be two right $\CC$-module categories (both with  identity endofunctors) and $\fun\colon \modcat \to \modcatN$ is a $\CC$-module functor. 
\begin{proposition}\label{prop:pullb-ideal}
Let $\modcat,\,\modcatN,\,\fun$ as above, and let $\ideal$ be a $\cat$-submodule category of $\modcatN$ closed under retracts then so is for the full subcategory 
\be\label{eq:fun-star}
\fun^*\ideal:=\{\,X\in\modcat\,|\, \exists Y\in\ideal\colon\,\fun X\cong Y\,\}\gp
\ee
\end{proposition} 
\begin{proof}
We first show that $\fun^*\ideal$ is a $\cat$-submodule category of $\modcat$.
Let $X\in\fun^*\ideal$ and $V\in\CC$. To show  that $X\otact V\in\fun^*\ideal$, we first observe the isomorphism $\fun(X\otact V)\cong \fun(X)\otact V$ which is the module structure of $\fun$. Since $\fun(X)\in \ideal$, we conclude from here that $\fun(X\otact V)$ is also in $\ideal$ and so $X\otact V$ is in $\fun^*\ideal$ as claimed.

To check that $\fun^*\ideal$ is closed under retracts, let $X'\in\modcat$ be a retract of $X$. This means that there exist morphisms $r\colon X\rightarrow X' $ and $\iota\colon X'\rightarrow X $ such that $r\circ\iota=\id_{X'}$. This implies $\fun(r)\circ\fun(\iota)=\id_{\fun X'}$, whence $\fun X'$ is a retract of $\fun X \in\ideal$ and hence $\fun X'\in\ideal$. This shows that $X'\in\fun^*\ideal$.
\end{proof}

But not only $\cat$-submodules can be carried through module functors. Assume that we have a module trace on  $\modcatN$ then one can construct a module trace on $\modcat$ using pull-back along a module functor~$\fun$.

\begin{proposition}\label{prop:pull-back}
Let $\t^{\modcatN}$ be a right module trace on a right $\CC$-module category $\modcatN$ and $\fun\colon \modcat \to \modcatN$ is a $\CC$-module functor, then the family of maps
\be\label{eq:trace-pull-back}
\t^{\modcat}_M(f) := \t^{\modcatN}_{\fun(M)}\bigl(\fun(f)\bigr),
\ee
for $M\in\modcat$ and $f\in\End_{\modcat}(M)$,
defines a right module trace on $\modcat$. 
\end{proposition}
\begin{proof}
Let $f\colon N\rightarrow M,\, g\colon M\rightarrow N$, a simple calculation yields
\begin{align*}
\t^\modcat_M(f\circ g)&=\t^\modcatN_{\fun(M)}\bigl(\fun(f\circ g)\bigr)\\
&=\t^\modcatN_{\fun M}\bigl(\fun(f)\circ\fun(g)\bigr)\\
&=\t^\modcatN_{\fun N}\bigl(\fun(g)\circ\fun(f)\bigr)\\
&=\t^\modcat_{N}\bigl(g\circ f\bigr)\gp
\end{align*}
Therefore, $\t^\modcat$ is a trace map on $\modcat$. It remains to check that the partial trace property holds. Let $f\in\End(P\otact V)$, then
\begin{align*}
\t^\modcat_{P\otact V}\bigl(f\bigr)&=\t^\modcatN_{\fun(P\otact V)}\bigl(\fun(f)\bigr)\\
&\stackrel{\text{cycl.}}{=}\t^\modcatN_{\fun(P)\otact V}\bigl(\sigma\fun(f)\sigma^{-1}\bigr)\\
&=\t^\modcatN_{\fun(P)}\bigl(\tr^r_V(\sigma \fun(f)\sigma^{-1}) \bigr)\\
&=\t^\modcatN_{\fun(P)}\bigl(\sigma\fun\bigl(\tr^r_V( f)\bigr) \sigma^{-1}\bigr)\\
&=\t^\modcat_P\bigl(\tr^r_V(f)\bigr)\gp
\end{align*}
Where $\sigma$ is the module structure of $\fun$. This shows that $\t^\modcat$ is a right module trace. 
\end{proof}

An example of  module functors is provided by the restriction functor for a pair $(H,B)$ where $H$ is a  Hopf algebra and $B$ is a right coideal subalgebra, i.e.\ $B$ is a subalgebra such that $\Delta(B)\subset B\tensor H$. 
First, $B\text{-mod}$ has a natural module structure over $\Hmod$ since for any $B$-module $V$ and $H$-module~$X$, we have 
$$
 V\otact X := V\otimes X
$$
 is a $B$-module via the restriction $\Delta|_B\colon B \to B\tensor H$. Let now $\cat=\Hmod$,  we show that the functor $\mathrm{Res}\colon \Hmod \to  B\text{-mod}$ is a right $\cat$-module functor, with the standard $\cat$-module structure on $\cat$. Indeed, for $X,Y\in\Hmod$, the $B$ actions on $\mathrm{Res}(Y\otimes X)$ and $\mathrm{Res}(Y)\otimes X$ agree: for $a\in B$  and $y\otimes x\in\mathrm{Res}(Y\otimes X)$ we have 
 $$
 a\cdot(y\otimes x)=\Delta|_B(a)(y\otimes x)=a_{(1)}y\otimes a_{(2)}x\gc
 $$
  where $a_{(1)}\in B$ and $a_{(2)}\in H$. Therefore,  
  $$
  \id\colon \; \mathrm{Res}(Y\otimes X)\rightarrow\mathrm{Res}(Y)\otimes X
  $$
   is $B$-linear and we can choose the identity map for the right module structure of $\mathrm{Res}$. 

Now, in the context of Proposition~\ref{prop:pullb-ideal}, we choose $\cat=\modcat=\Hmod$ and $\modcatN=\Bmod$, and the module functor $\fun=\Res$. Then note that $\Bpmod\subset \modcatN$ is 
 also a right module category over $\Hmod$.
 We then have $\Res^* \ideal=\ideal_B$, as defined in~\eqref{eq:IA-def}, is a right tensor ideal in $\Hmod$.

Proposition~\ref{prop:pull-back} turns out to be useful in constructing module traces~\cite{FG} for a large class of tensor ideals in $\Hmod$ associated with Hopf subalgebras in a pivotal Hopf algebra~$H$. Note that $\fun = \mathrm{Res}$ is also  a faithful functor, i.e.\ $\fun(f)=0$ iff $f=0$ in this case. 
In particular, we have

\begin{corollary}\label{prop:pullbackunimod}
Let $(H,\pivot)$ be a finite-dimensional pivotal Hopf algebra and $A$ be a unimodular Hopf subalgebra in $H$ such that $\pivot\in A$. Then, there is a
left/right  modified trace on the tensor ideal $\ideal_A\subset\Hmod$ from~\eqref{eq:IA-def} given by the pull-back~\eqref{eq:trace-pull-back} for the choice $\modcat=\ideal_A$ and $\modcatN=\Apmod$.
\end{corollary}
\begin{proof}
Since $A$ is a unimodular Hopf algebra, then there is a (non-degenerate) modified trace or a module trace $\t^A$ on $\Apmod$ considered as a module category over $\Amod$, see ~\cite{BBG} where this trace is expressed via the integral of $A$. Recall that $\Apmod$ is  not not only a tensor ideal in $\Amod$ but
also a module category over $\Hmod$. We show that the family of trace maps~$\t^A$ is  also a module trace over $\Hmod$. We only need to show~\eqref{eq:rpartial} where $V\in\Hmod$, $\Sigma=\Id$ and with the right partial trace map \eqref{E:PartialLRtrace}  defined using $\tev$ and $\coev$ duality maps in $\Hmod$. 
 But this right partial trace property follows from the assumption that
 $A$ and $H$ share the same pivot, and so the pivotal structure on $\Hmod$ agrees with the one in $\Amod$, in particular so do the right partial traces.

As noted above, by Proposition~\ref{prop:pullb-ideal} we have that $\ideal_A=\mathrm{Res}^*(\Apmod)$ is a tensor ideal in $\Hmod$. Then by Proposition~\ref{prop:pull-back} for the choice $\cat=\Hmod$, $\modcat=\ideal_A$ and $\modcatN=\Apmod$, there is indeed a non-zero (recall that here $\fun=\Res$ is faithful) modified trace on $\ideal_A$ given by the pull-back~\eqref{eq:trace-pull-back}.
\end{proof}

 \subsection{Reduction Lemma}
 We now formulate a so-called Reduction Lemma which is a generalisation of a similar lemma from~\cite{BBG}. Recall that an abelian category $\CC$ is said to be finite pivotal if it is finite as an abelian category, rigid monoidal with tensor product linear and exact on each component, and it has a simple tensor unit and a pivotal structure. 
 
 \begin{lemma}[{\em Reduction Lemma}]\label{lem:red}
 Given a finite pivotal category $\CC$ 
  and  a $\CC$-module endo\-functor $\Sigma$ on $\proj(\CC)$, a $\Sigma$-cyclic trace map 
$\{\t_P\colon\cat(P,\Sigma P)\rightarrow\ok\}_{P\in\proj(\CC)}$ satisfies the right partial trace property \eqref{eq:rpartial} if and only  if the following equation holds:
\begin{equation}\label{eq:tGG-tG}
\t_{\prog\otimes\prog}(f)=\t_\prog\bigl(\mathrm{tr}^{\Sigma}_\prog(f)\bigr)
\end{equation}
for all   $f\colon\prog\otimes \prog\rightarrow \Sigma(\prog\otimes \prog)$ and for a projective generator $\prog\in\CC$.
 \end{lemma}
 
 \begin{proof}
 First consider the case when  $P, P'$ are projective. Then, from~\cite[Lem.\,2.2]{BBG}, we have morphisms $a_i\colon\prog\rightarrow P,\,b_i\colon P\rightarrow\prog$ and $a_i'\colon\prog\rightarrow P',\,b_i'\colon P'\rightarrow\prog$ such that
 \be\label{eq:idP-idP}
\id_P=\sum_{i\in I}a_i\circ \id_{\prog}\circ b_i\ ,\qquad
  \id_{P'}=\sum_{{j\in J}}a'_{j}\circ \id_{\prog}\circ b'_{j}\gc
 \ee
 for some finite sets $I$ and $J$.
Using these decompositions of identity maps, a simple computation yields
$$
\t_{P P'} (f)\; =\; 
\t_{P P'}\left(
\ipic{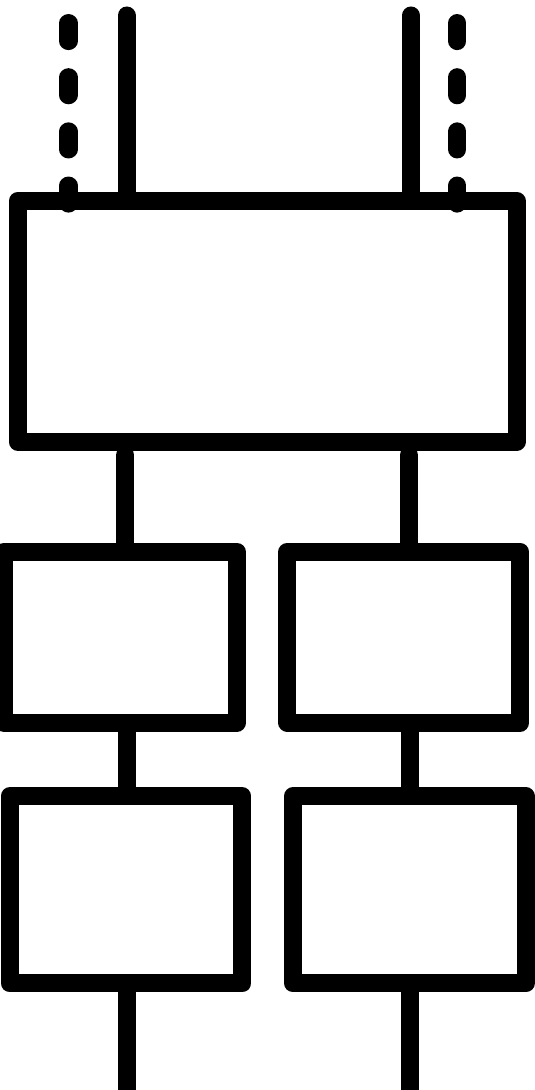}{.21} 
\put(-16,12){\scriptsize $f$}
\put(-27,-7){\scriptsize $a_i$} \put(-11,-7){\scriptsize $a_{j}'$}
\put(-26,-23){\scriptsize $b_i$} \put(-11,-23){\scriptsize $b_j'$}
 \right)
  \; \stackrel{\text{cycl.}}{=} \; 
 \t_{\prog\prog}\left(
\ipic{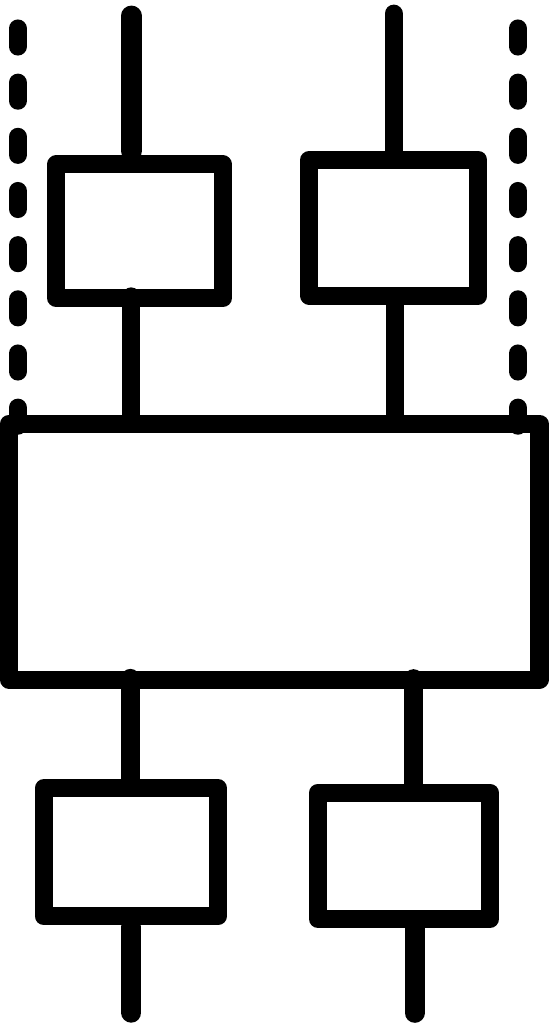}{.21} 
\put(-28,15){\tiny $b_i$} \put(-13,15){\tiny $b_j'$}
\put(-19,-5){\scriptsize $f$}
\put(-29,-22){\tiny $a_i$} \put(-13.5,-22){\tiny $a_j'$}
 \right)
    \; \stackrel{\eqref{eq:tGG-tG}}{=} \; 
   \t_{\prog}\left(
\ipic{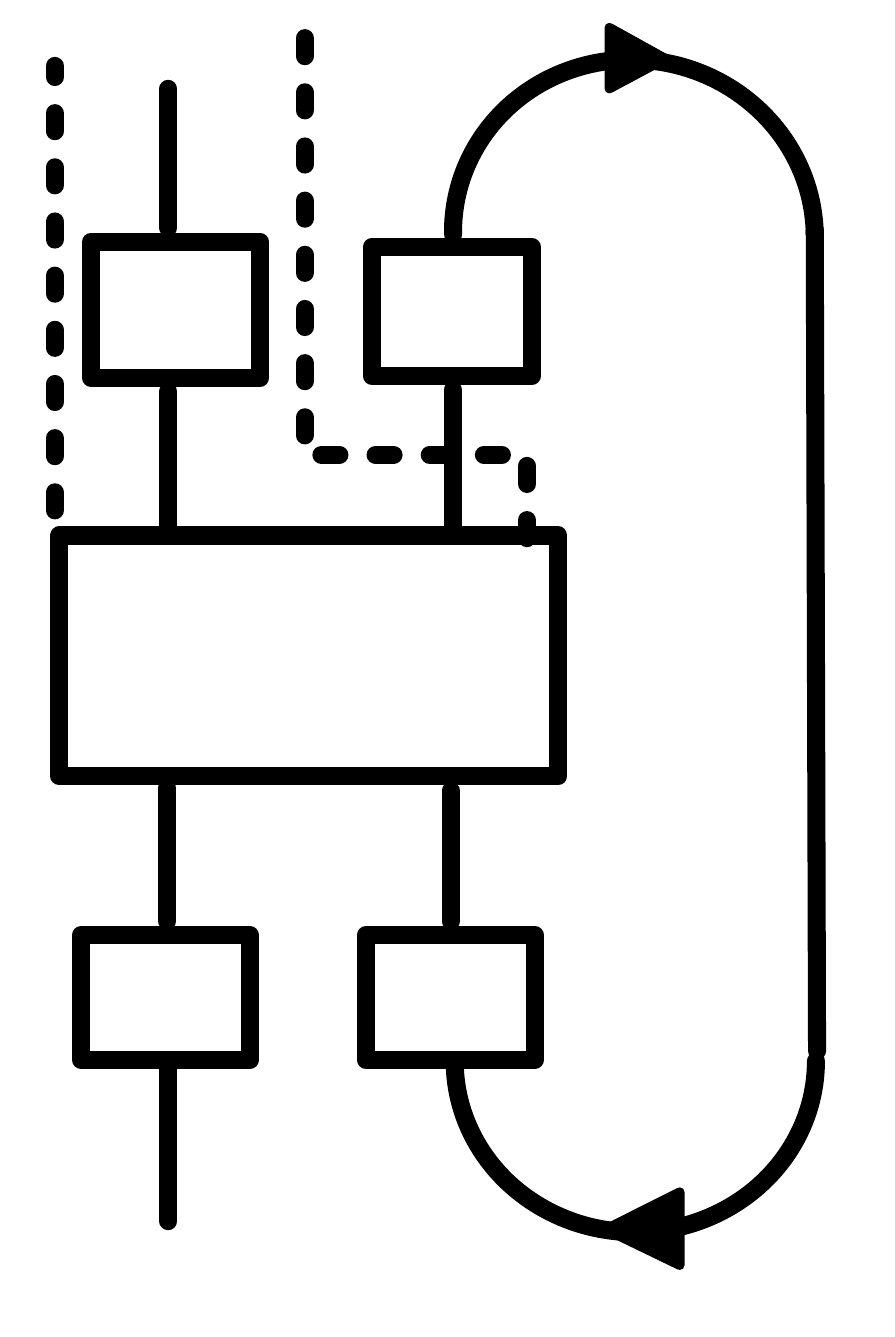}{.21}
\put(-45.5,19){\tiny $b_i$} \put(-31,19){\tiny $b_j'$}
\put(-37,-3){\scriptsize $f$}
\put(-48,-22){\tiny $a_i$} \put(-31.5,-21.5){\tiny $a_j'$}
 \right)
   \; = \; 
   \t_P\bigl(\tr^\Sigma_{P'}(f)\bigr)\gc
$$
where the tensor product signs were omitted for brevity and the sums are assumed over the repeated indices $i\in I,\,j\in J$. In the third equality, besides \eqref{eq:tGG-tG} we use \eqref{diag:natstr}. In the last step, standard manipulations with duals were used to move maps around the loop, and then the $\Sigma$-cyclicity of $\t$
was applied. 

\begin{figure}
\scalebox{1}[1.15]{
\scriptsize \centerline{\def\svgwidth{9cm}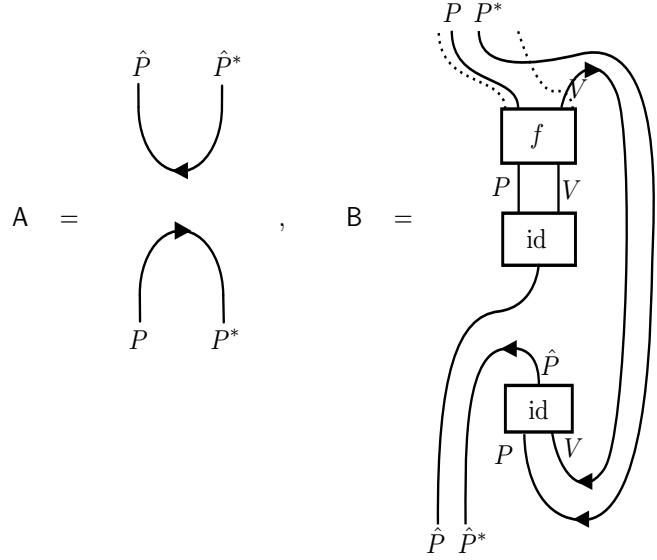}
}
\caption{Morphisms $\A$ and $\B$. }
\label{fig:F}
\end{figure}

It remains to show that the result holds for $P$ projective and $V\in\mathcal{C}$. This follows from the case above. Indeed, set $\tP=P\otimes V$, then given $f\colon\tP\rightarrow\Sigma\tP$ define the auxiliary morphisms $\A,\B$ as presented graphically in Figure \ref{fig:F}. 
Then the following series of equalities hold
\begin{align*}
\t_P(\tr^\Sigma_V(f)) &= \t_P(\tr^\Sigma_{P^*}(\B\circ \A))\\
&= \t_{P\otimes P^*}(\B\circ \A)\\
&= \t_{\tP\otimes\tP^*}(\Sigma\A\circ\B)\\
&= \t_{\hat{P}}(\tr^\Sigma_{\tP^*}(\Sigma\A\circ\B))\\
&= \t_{P\otimes V}(f)\gp
\end{align*}
\end{proof}

  \subsection{Compatibility with duality}
We recall that in any $\Bbbk$-linear pivotal category $\catd$ we have the following duality isomorphisms:
\begin{equation} \label{E:duality_iso_r}
\begin{split}
\begin{array}{rcl}
d^{\cap}\colon\; &\Hom_\catd(W,U\otimes V) \xrightarrow{\sim} \Hom_\catd(W\otimes V^*, U)\\
&f\mapsto (\id_U\otimes \tev_V)\circ (f\otimes \id_{V^*})\ 
\end{array}\
, \qquad
 \ipic{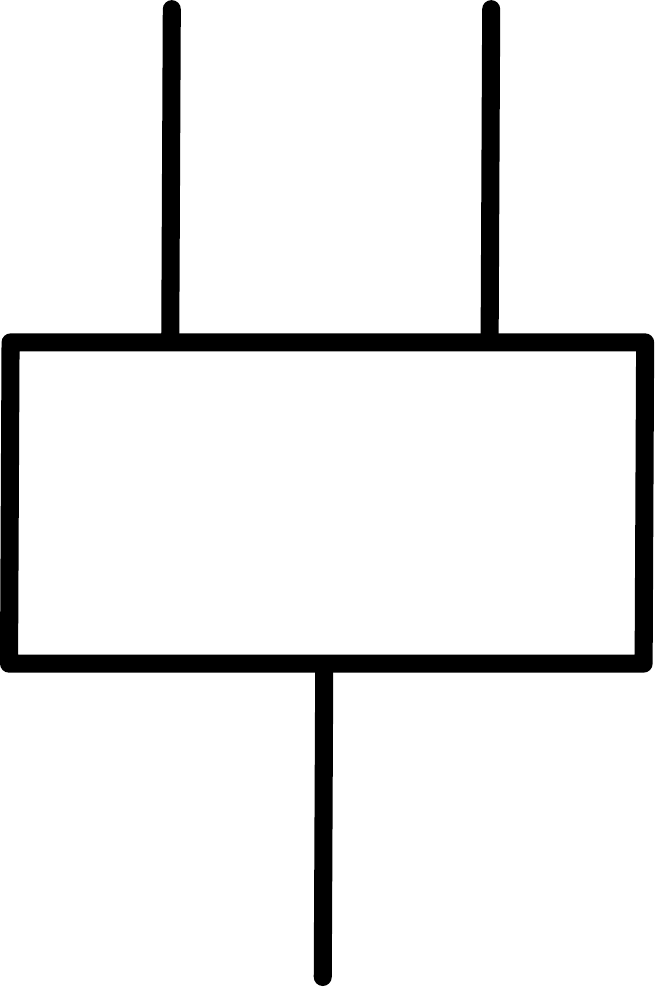}{.15} 
\put(-17,-3){\footnotesize $f$}
 \put(-25,24){\scriptsize $U$} \put(-11,24){\scriptsize $V$} 
\put(-18,-29){\scriptsize $W$} 
  \; \mapsto \; 
\ipic{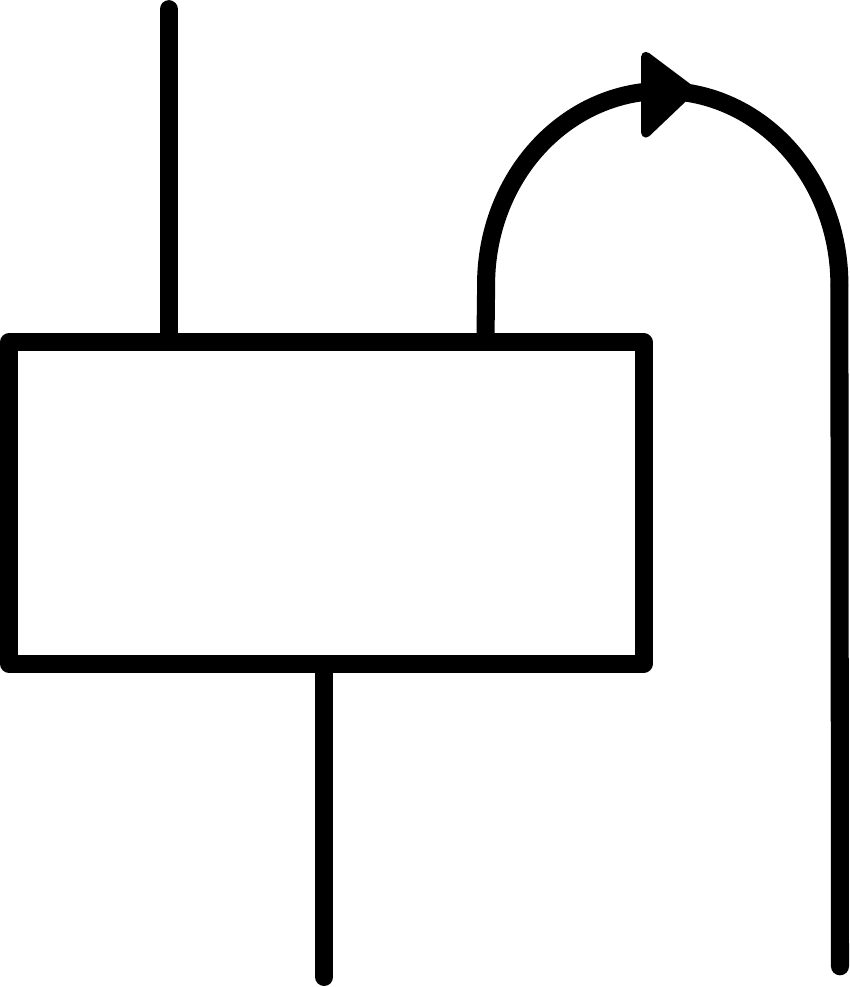}{.15} 
\put(-25,-3){\footnotesize $f$}
 \put(-33,24){\scriptsize $U$} 
\put(-28,-29){\scriptsize $W$}  \put(-4,-29){\scriptsize $V^*$} 
\put(-28,-35){\ }
\\
\begin{array}{rcl}
d_{\cup}\colon\; &\Hom_\catd(U\otimes V, W) \xrightarrow{\sim} \Hom_\catd(U,W\otimes V^*)\\
&f\mapsto(f\otimes \id_{V^*})\circ (\id_U\otimes \coev_V) \ 
\end{array}\
,\qquad
 \ipic{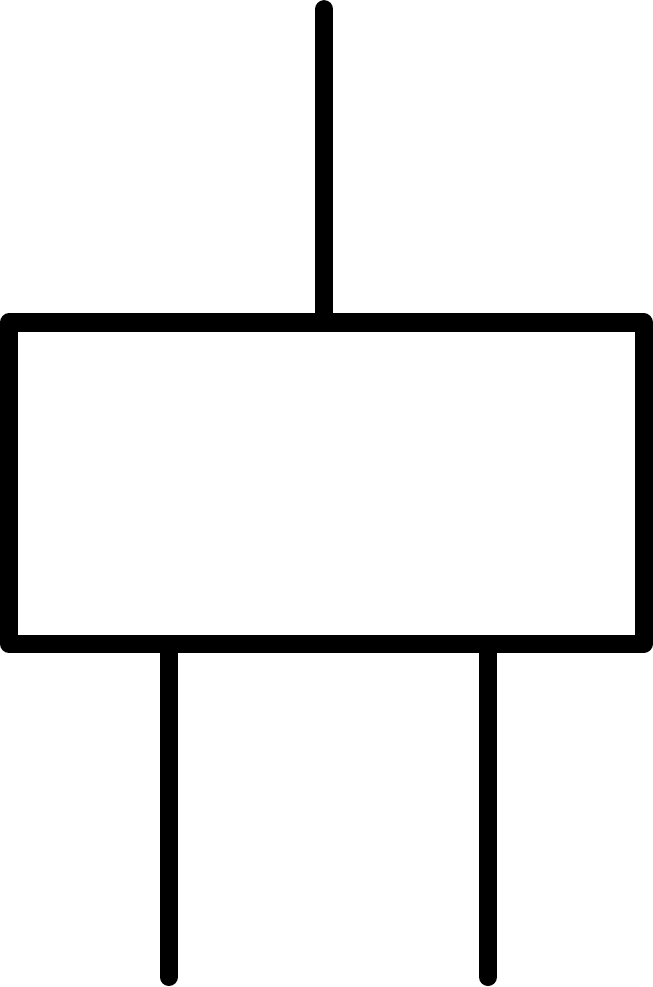}{.15} 
\put(-17,-2){\footnotesize $f$}
\put(-18,24){\scriptsize $W$} 
 \put(-25,-29){\scriptsize $U$} \put(-11,-29){\scriptsize $V$} 
  \; \mapsto \; 
\ipic{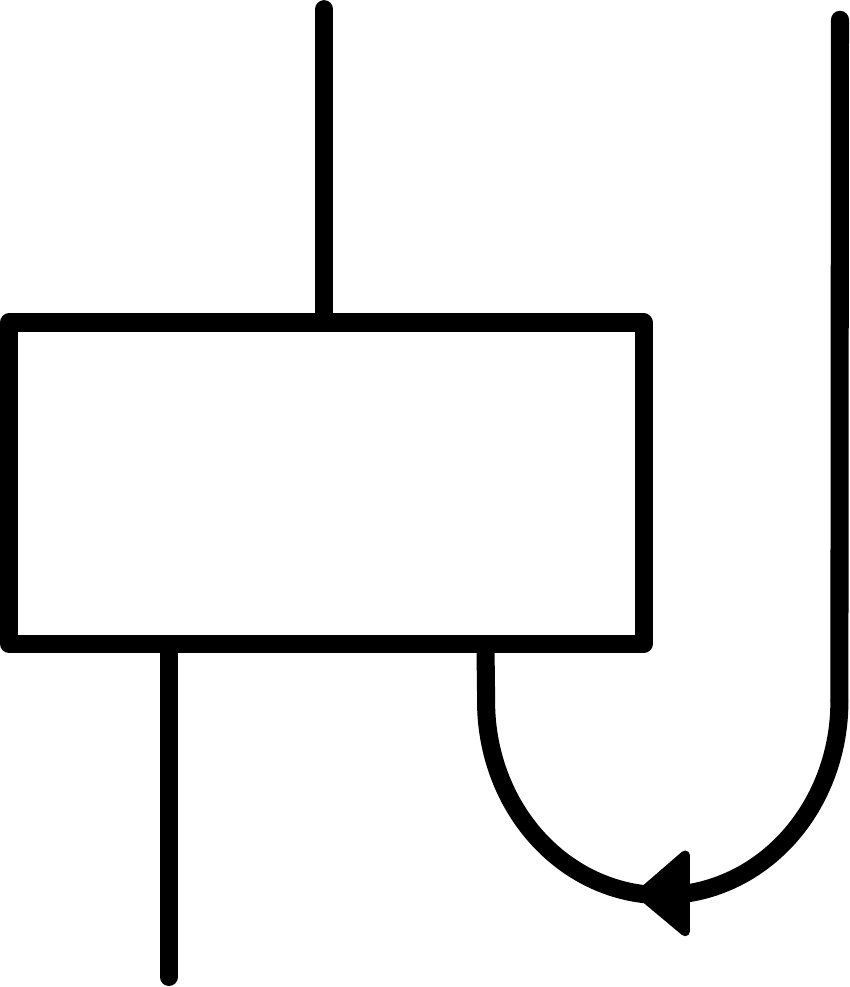}{.15} 
\put(-25,-2){\footnotesize $f$}
\put(-28,24){\scriptsize $W$}  \put(-4,24){\scriptsize $V^*$} 
 \put(-33,-29){\scriptsize $U$} 
 \end{split}
\quad .
\end{equation}
One can similarly define another pair of isomorphisms ${}^{\cap}d$ and ${}_{\cup}d$ with (co)evaluation morphisms on the left instead of the right side, see~\cite[Sec.\,3]{BBG}. 

We show next that under some extra assumptions the partial trace property \eqref{eq:rpartial} is equivalent to compatibility of a trace map with the duality isomorphisms $d^{\cap}$ and $d^{\cup}$.  
Let $\CC$ be a  $\Bbbk$-linear pivotal category and $\Sigma$ a $\CC$-module endofunctor of a module category $\modcat$. 
We call a trace map $\{ \t_P\colon\CC(P,\Sigma P)\rightarrow\Bbbk\}$ on $(\modcat,\Sigma)$ 
{\it compatible with  duality on the  right}
 if the following
diagram commutes, for all $U,W\in \modcat,V\in\catd$,
\begin{equation}\label{rcomp}
\xymatrix@R=20pt@C=40pt@W=10pt@M=10pt{
\Hom_{\CC}(U\otimes V,W) \times \Hom_{\CC}(W,\Sigma(U\otimes V)) \ar[r]^{\qquad\quad\circ}\ar@<-45pt>[dd]_{d_\cup}\ar@<25pt>[d]^\sigma&
\Hom_{\CC}(U\otimes V,\Sigma(U\otimes V))\ar@<-0pt>[d]^{\t_{U\otimes V}}\\
\qquad\qquad\qquad\qquad\,\Hom_\CC(W,\Sigma(U)\otimes V)\ar@<25pt>[d]^{d^\cap}& \Bbbk\\
 \Hom_\CC(U,W\otimes V^*) \times  \Hom_\CC(W\otimes V^*, \Sigma(U)) \ar[r]^{\quad\qquad\qquad\circ}& \Hom_{\CC}(U,\Sigma(U)) \ar[u]_{\t_U}
}
\end{equation}
where  the $\sigma$ map stands for post-composition with the $\CC$-module isomorphism  $\sigma_{U,V}$.

We can similarly define a trace map compatible with  duality on the left using analogous diagram to~\eqref{rcomp} involving now  the pair  of isomorphisms  ${}^{\cap}d$ and ${}_{\cup}d$  and composition with the left  module structure $\sigma$.

 \begin{theorem}\label{thm:compdual}
 Let $\CC$ be a finite pivotal category and $\Sigma\in\End_\CC\bigl(\proj(\CC)\bigr)$. A trace map on $\bigl(\proj(\CC),\Sigma\bigr)$ is compatible with duality on the right if and only if it satisfies the right partial trace property.
 \end{theorem}
 \begin{proof}
 The proof is straightforward and essentially follows the one in~\cite[Thm.\,3.3]{BBG}.
 \end{proof}

\section{$G$-categories}\label{sec:Gcat}
\begin{definition}
 Given a group $\group$ and $\CC$ a $\ok$-linear monoidal category. $\CC$ is then called a \textit{$\group$-graded category}, or simply a \textit{$\group$-category}, if there is a family of full subcategories $(\CC_x)_{x\in\group}$ such that the following holds:
 \begin{enumerate}
  \item $\mathbb{1}\in\CC_1$,
  \item $\CC_x\neq 0$ for all $x\in\group$ and we have a decomposition $\CC=\bigoplus_{x\in\group}\CC_x$,
  \item  for all $x,y\in\group$, and $V_x\in\CC_x$, $V_y\in\CC_y$, if $\CC(V_x,V_y)$ is non-zero then $x=y$,
  \item $\CC_x\otimes\CC_y\subset\CC_{xy}$.
 \end{enumerate}  
\end{definition}

An easy consequence of this definition is that in a $\group$-category duality and inverses in $\group$ are related as follows, if $V\in\CC_x$ then $V^*\in\CC_{x^{-1}}$. This makes the double dual functor a graded endofunctor, that is $V^{**}\in\CC_x$ for all $V\in\CC_x$. Hence, a pivotal structure on a rigid $\group$-category is a family of natural isomorphisms 
$$
\delta^x\colon\Id_{\CC_x}\rightarrow(-)^{**}_{\CC_x}
$$
 indexed by ${x\in\group}$ and satisfying the standard axioms. 
 
 \begin{definition}\label{def:modGcat}
 Let $\CC$ be a $\group$-category. A $\CC$-module category $\modcat$ is called $G$-graded if there is a family of subcategories $(\modcat_x)_{x\in\group}$ such that
 \begin{enumerate}
  
  \item $\modcat_x\neq 0$ for all $x\in\group$ and we have a decomposition $\modcat=\bigoplus_{x\in\group}\modcat_x$,
 
  \item for all $x,y\in\group$, $\modcat_x\otact\CC_y\subset\modcat_{xy}$.
 \end{enumerate}  
 \end{definition} 
 
An important example of $\group$-graded module categories is the projective ideal of $\CC$. Indeed, we  have a decomposition
 $$
 \proj(\CC)=\bigoplus_{x\in\group}\proj(\CC_x).
 $$
  Furthermore, if $\CC$ is finite, then it follows that so are $\group$ and $\CC_x$ for all $x\in\group$ . In this case a progective generator of $\CC$ must be of the form $\prog=\bigoplus_{x\in \group}\prog_x$, where $\prog_x$ is a projective generator of $\CC_x$.  
  
By a $\group$-endocategory we mean a pair $(\CC,\Sigma)$ where $\CC$ is $\group$-graded and $\Sigma$ is a graded endofunctor. According to Lemma~\ref{lem:red} a trace map on $\bigl(\proj(\CC),\Sigma\bigr)$ will be a module trace only if the equation \eqref{eq:tGG-tG} is satisfied. From the discussion above this reduces to checking similar equations involving the projective generators of each $\CC_x$. However, similar techniques can be used to prove a similar result under a weaker requirement than $\CC$ being finite. This result can be used in the case $\group$ is not finite, which will be important for  examples we consider below.    
\begin{lemma}[Graded reduction]\label{lem:gred}
Let $\CC$ be a $\group$-category, $\Sigma\in\End_\CC(\proj(\CC))$ graded and $t$ a trace map on $(\proj(\CC),\Sigma)$. Further, assume that $\CC_x$ is finite for all $x\in\CC$, and $\prog_x$ is a projective generator of $\CC_x$, then $t$ is a module trace if and only if  the equation holds:
\be\label{eq:gtGG}
t_{\prog_x\otimes\prog_y}(f)=t_{\prog_x}\bigl(\tr^\Sigma_{\prog_y}(f)\bigr)
\ee
for all $x,y\in\group$ and for all $f\colon\prog_x\otimes\prog_y\rightarrow\Sigma(\prog_x\otimes\prog_y)$.
\end{lemma}   
\begin{proof}
The proof is similar to the one used in \cite{Phu} for the case $\Sigma=\Id$. It is similar to the proof of Lemma~\ref{lem:red}, and is hence only sketched.
 It is clear that \eqref{eq:gtGG} is a necessary condition. It remains to show that it is sufficient.
 We begin  by remarking that any map between objects in $\CC$ can be seen as a sum of homogeneous maps, and hence it suffices to prove the statement for homogeneous objects.
   We check \eqref{eq:rpartial} for $V\in\proj(\CC_y),P\in\proj(\CC_x)$. As before we use decompositions of $\id_V,\id_P$ via $\prog_y,\prog_x$ respectively, to get the desired result.
 To conclude \eqref{eq:rpartial} for non projectives of $\CC_y$ we use a similar trick as the one used in the proof of Lemma~\ref{lem:red} using the maps in Figure~\ref{fig:F}, now with $V\in\CC_y$.
\end{proof}

To finalise this section, we remark that a module trace on a $\group$-module category is a module trace on the underlying module category. Hence, similar statements to the ones shown in  Section~\ref{sec:endoc} hold here as well. In particular, the pull-back construction in Section~\ref{sec:pullback} extends to $\group$-module categories, and Theorem~\ref{thm:compdual} holds for finite type $\group$-categories as well.

\section{Hopf group-coalgebras}\label{sec:HopfGC}

It is known that finite dimensional Hopf algebras provide finite tensor categories . In this section we introduce Hopf $\group$-coalgebras as the corresponding algebraic object to the theory of $\group$-categories. A thorough account of the theory of Hopf group-coalgebras was done in~\cite{HopfGC}. 

\begin{definition}
	Let $\group$ be a group. \textit{A Hopf $\group$-coalgebra $H$} is a family of algebras 
	$$
	H = \{(H_x,m_x,1_x)\}_{x\in\group}
	$$
	 together with the family of algebra maps 
	$$
	\Delta_{x,y}\colon H_{xy}\rightarrow H_x\otimes H_y\gc \qquad \varepsilon\colon H_1\rightarrow \ok
	$$ 
	and linear maps 
	$${S_x\colon H_x\rightarrow H_{x^{-1}}}$$
	 such that the following holds
	\begin{description}
	 \item[Coassociativity] for all $x,y,z\in\group$:
	 $$
	 (\Delta_{x,y}\otimes\id_{H_z})\circ\Delta_{xy,z}=(\id_{H_x}\otimes\Delta_{y,z})\circ\Delta_{x,yz} 
	 $$
	 \item[Counitality] for all $x\in\group$: 
	 $$
	 (\id_{H_x}\otimes\varepsilon)\circ\Delta_{x,1}=\id_{H_x}=(\varepsilon\otimes\id_{H_x})\circ\Delta_{1,x}
	 $$
	 \item[Antipode law] for all $x\in\group$:
	  $$
	  m_x\circ(S_{x^{-1}}\otimes\id_{H_x})\circ\Delta_{x^{-1},x}= 1_x\cdot \varepsilon=m_x\circ(\id_{H_x}\otimes S_{x^{-1}})\circ\Delta_{x,x^{-1}}\gp
	  $$
	\end{description}
	We say that a Hopf $\group$-coalgebra is of \textit{finite type} if  $H_x$ is finite dimensional for all $x\in\group$.
\end{definition}  
Many results from the theory of Hopf algebras have an analog in this setting, see~\cite{HopfGC}. In particular,
\begin{proposition}\label{prop:antialg}
Let $H$ be a Hopf $\group$-coalgebra.
\begin{enumerate}
\item For $x\in\group$, the map $S_x$ is an antialgebra map,
\item For $x,y\in\group$,  the equality $\Delta_{y^{-1},x^{-1}}\circ S_{xy}=\tau_{H_{x^{-1}},H_{y^{-1}}}\circ (S_x\otimes S_y)\circ\Delta_{x,y}$ holds, where~$\tau$ is the flip map in $\vect$.

\item $\varepsilon\circ S_1=\varepsilon$.
\end{enumerate}
\end{proposition}
It is easy to see that if $H$ is a Hopf $\group$-coalgebra then $H_1$ is a Hopf algebra and all $H_x$ are $H_1$-comodules.

\begin{notation}[Sweedler]
In the spirit of Sweedler's notation for Hopf algebras we write $\Delta_{x,y}(a)=a_{(1,x)}\otimes a_{(2,y)}$.
\end{notation}

\begin{example}\label{ex:quotHGC}
\mbox{}
\begin{enumerate}
\item Any Hopf algebra is a Hopf $G$-coalgebra for the trivial group $G=\langle1\rangle$. 

\item
Let $H$ be a Hopf algebra, and a Hopf subalgebra $C\subset Z(H)$. Consider the group of characters $\group=Alg(C,\ok)$. For $f\in G$, define  the projection algebra map
$$
p_f\colon\; H\twoheadrightarrow H\big/\bigl(z-f(z)1\bigr)_{z\in C}=:H_f\gp
$$
We can then introduce the coproduct maps 
\begin{align}
&\Delta_{f_1,f_2}\colon\; H_{f_1 f_2} \to H_{f_1}\tensor H_{f_2}\gc \quad f_1,f_2\in G\gc\notag\\
&\Delta_{f_1,f_2}\bigl(p_{f_1f_2}(a)\bigr):=p_{f_1}(a_{(1)})\otimes p_{f_2}(a_{(2)})\label{eq:Delta-C}
\end{align}
 and antipode maps
 \be\label{eq:S-C}
 S_{f_1}\bigl(p_{f_1}(a)\bigr):=p_{f_1^{-1}}\bigl(S(a)\bigr)\gp
 \ee
 We first check that they  are well-defined.
  For $f_1,f_2\in\group$ and $z\in C$ we have  
\begin{align*}
	\Delta\bigl(z-f_1 f_2(z)1\bigr)&=z_{(1)}\otimes z_{(2)}-f_1\otimes f_2\bigl(z_{(1)}\otimes z_{(2)}\bigr)1\otimes 1\\
	&= \bigl(z_{(1)}-f_1(z_{(1)})1\bigr)\otimes z_{(2)}+z_{(2)}\otimes(z_{(2)}-f_2(z_{(2)})1)
	\intertext{and also }
	S\bigl(z-f_1(z)1\bigr)&=S(z)-f_1(z)1\\
	&= S(z)-f_1^{-1}\bigl(S(z)\bigr)1
\end{align*}  
and hence the maps~\eqref{eq:Delta-C} and ~\eqref{eq:S-C} are indeed well-defined. Notice also that $\varepsilon\in\group$ is the unit and $\varepsilon(H_\varepsilon)=0$, whence $\varepsilon $ descends to a map on $H_\varepsilon$. The family $\{H_f\}_{f\in\group}$ with the maps defined above is a Hopf $\group$-coalgebra, since the Hopf $\group$-coalgebra axioms follow from the Hopf algebra axioms for $H$ projected onto the quotients.
\end{enumerate}
\end{example}

A module over a Hopf $\group$-coalgebra $H$ is a vector space $V=\bigoplus_{i=1}^nV_i$ 
 where, for some $x_i\in\group$, $H_{x_i}$ acts on $V_i$. The category $\Hmod$ of modules over $H$ is then $G$-graded, with $(\Hmod)_x=H_x\text{-mod}$. Furthermore, if $H$ is of finite type then $\Hmod$ is finite in each degree, i.e.\ $(\Hmod)_x$ is a finite category for each $x\in G$.
 
 \subsection{Pivotal structure}
 We call $\pivot=(\pivot_x)_{x\in\group}\in\prod_{x\in\group}H_x$ a \textit{$\group$-group-like} element if 
 $$
 \Delta_{x,y}(\pivot_{xy})=\pivot_x\otimes\pivot_y \qquad \text{and} \qquad \varepsilon(\pivot_1)=1\gp
 $$
  As in the Hopf algebra case, the set of $\group$-group-like elements has the structure of a group with 
  $\pivot^{-1}_x=  S_{x^{-1}}(\pivot_{x^{-1}})$.

 \begin{definition} 
 A \textit{pivotal Hopf $\group$-coalgebra} is a pair $(H,\pivot)$, where $H$ is a Hopf $\group$-coalgebra and $\pivot$ a $\group$-group-like element such that $S_{x^{-1}}S_x(a)=\pivot_x\cdot a\cdot\pivot_x^{-1}$.
 \end{definition}
 
 \begin{example}
 If $H$ is a pivotal Hopf algebra the constuction from Example~\ref{ex:quotHGC} (2) yields a pivotal Hopf $\group$-coalgebra.
 \end{example}
 Modules over a pivotal Hopf $\group$-coalgebra form a pivotal category. Indeed, let $V\in \Hxmod{x}$, then we can endow the vector space $V^*=\Hom_\ok(V,\ok)$ with the structure of $H_{x^{-1}}$-module via the action $(h\cdot f)(a)=f\bigl(S_{x^{-1}}(h)a\bigr)$ for $f\in V^*, h\in H_{x^{-1}}, a\in H_x$. In this case the evaluation and coevaluation maps witnessing left and right duality are given by 
 
 \be\label{eq:leftevs}
 \ev_V(f\otimes a)=f(a)\gc\qquad \coev_V(1)=\sum_i v_i\otimes v_i^*\
 \ee
 \be\label{eq:rightevs}
 \widetilde{\ev}_V(a\otimes f)=f(\pivot_xa)\gc \qquad  \widetilde{\coev}_V(1)=\sum_iv_i^*\otimes\pivot_x^{-1}v_i 
 \ee
  where 
    $V\in \Hxmod{x}$, $a\in V$, $f\in V^*$, and $\{v_i\}$ is a basis of $V$.

 \subsection{Module endofunctors}
 We note that if $H$ is a Hopf $\group$-coalgebra then $H_1$ coacts on $H_x$ via $\Delta_{x,1}$ making it a  right $H_1$-comodule-algebra, that is an algebra with a multiplicative $H_1$-comodule structure.  
 For $a\in H_x$ and $\alpha\in H_1^* $, we define the right hook action 
 \begin{align}\label{eq:rhook}
  &R(\alpha)\colon\; H_x \to H_x\gc\qquad x\in G\gc\notag\\
 &R(\alpha)(a):=a\leftharpoonup\alpha=\alpha(a_{(1,1)})a_{(2,x)}\gp
 \end{align}
  Clearly, $R(\alpha)$ is a right comodule map but we can say more. 
 
 \begin{proposition}\label{prop:hookcomodule}
 Let $\{H_x\}_{x\in\group}$ be a Hopf $\group$-coalgebra. A family of maps $\varphi=(\varphi_x)_{x\in\group}\in\prod_{x\in\group}\End_\ok(H_x)$ satisfies the relation
 \be\label{eq:defralpha}
  \Delta_{x,y}\circ\varphi_{xy}=(\varphi_x\otimes\id_y)\circ\Delta_{x,y}\qcq\forall x,y\in\group
 \ee  
 if and only if there exists $\alpha\in H_1^*$ such that $\varphi_x=R(\alpha)$ for all $x\in\group$.
 Furthermore, the space of such families of maps carries an algebra structure isomorphic to $(H_1^*)^\mathrm{op}$.
 \end{proposition}
 \begin{proof}
  A straightforward calculation shows that $R(\alpha)$ satisfies the relation \eqref{eq:defralpha}. 
 Assume now that $\varphi$ satisfies this relation. In particular, $\Delta_{1,y}\circ\varphi_y=(\varphi_1\otimes\id_y)\Delta_{1,y}$. Applying $\varepsilon\otimes\id_y$ on both sides yields $\varphi_y=(\varepsilon\circ\varphi_1\otimes\id_y)\Delta_{1,y}=R(\varepsilon\circ\varphi_1)$.
 
 To finish notice that $R(\alpha)\circ R(\beta)=R(\beta\alpha)$ for all $\alpha,\beta\in H_1^*$.
 \end{proof}
 To see the relevance of  this statement in building right module endofunctors on the category $\Hmod$ we first need some preparations.
 
 Given an algebra $A$ and an algebra endomorphism $\varphi$, we define the {\em twist functor} 
\begin{align}\label{defeq:twistftor}
\varphi_*\colon \; &\Amod\rightarrow \Amod\notag\\
& V \mapsto {}_\varphi V
\end{align}
where ${}_\varphi V$ is the module structure on $V$ twisted by $\varphi$, i.e.\  the underlying vector space is $V$ and the action is $a\cdot_\varphi v=\varphi(a)\cdot v$. On morphisms, the functor $\varphi_*$ is identity.
It is straightforward to see that $\varphi_*$ is an additive and exact functor, i.e.\ it respects direct sums and exact sequences.

In the $\group$-graded case we are interested in functors that are twists degreewise. That is, for all $x\in\group$ we have an algebra endomorphism $\varphi_x\colon H_x\rightarrow H_x$ and for $V\in \Hxmod{x}$ the image under the functor is given by $(\varphi_x)_*(V)$. We have the following characterization.

\begin{proposition}\label{prop:mod-twisthook}
Let $H$ be a Hopf $\group$-coalgebra and for every $x\in\group$ let $\varphi_x\in Alg(H_x,H_x)$. The functor given by $(\varphi_x)_*\colon\Hxmod{x}\rightarrow \Hxmod{x}$ is a module endofunctor if and only if $\Delta_{x,y}\circ\varphi_{xy}= (\varphi_x\otimes\id_y)\circ\Delta_{xy} $ holds. In particular, in this case there exists  $\nu\in H_1^*$ group-like such that $\varphi_x=R(\nu)$ 
 for all $x\in\group$.
\end{proposition}
\begin{proof}
Applying the module condition to the tensor product of regular representations $H_x\otimes H_y$ we get that the action of $H_{xy}$ on $(\varphi_{xy})_*(H_x\otimes H_y)$, given by $\Delta_{x,y}\circ \varphi_{xy}$, must agree with the action on $(\varphi_x)_*(H_x)\otimes H_y$, given by $(\varphi_x\otimes\id_y)\Delta_{x,y}$.
The last statement follows from Proposition~\ref{prop:hookcomodule}, where $\nu=\varepsilon\circ\varphi_1$ and hence group-like.
\end{proof}
 
 \begin{remark}\label{rem:cophook}\mbox{}
 \begin{enumerate}
 \item
 For a Hopf $\group$-coalgebra $H$ and $\nu\in H_1^*$, we can define similarly the left hook action  as 
 \begin{align}
 &L(\nu)\colon \; H_x \to H_x\gc \notag\\
 &L(\nu)(a):=\nu \rightharpoonup a = (\id_{H_x}\otimes \nu)\circ \Delta_{x,1}(a)\gp
 \end{align}
  Notice that this is nothing but the right hook action on the Hopf $\group^{\mathrm{op}}$-coalgebra $H^{\mathrm{cop}}$. Hence, similar results apply to $L(\nu)$.
 
 \item  We  also note that if $\alpha\in H_1^{*}$ is group-like it then provides a 1-dimensional invertible object $X_\alpha$ in $\Hmod$. In this case, $R(\alpha)$ is an algebra automorphism, and  the corresponding twist endofunctor $R(\alpha)_*$ can be written as taking tensor product with $X_\alpha$ from the left, i.e.
 $$
 R(\alpha)_* = X_\alpha\tensor -
 $$
 Similarly, $L(\alpha)_*$ corresponds to tensoring with $X_\alpha$ from the right side.
 Tensoring with invertible objects provides an important class of endofunctors, recall also  Examples~\ref{sec:examples}~(1) above.
 \end{enumerate}
 \end{remark}

 \subsection{Integrals and modulus}
 Here, we review the theory of integrals and cointegrals of a Hopf $G$-coalgebra of finite type.
 
 \newcommand{\comodulus}{\boldsymbol{a}}
 \begin{definition}\label{def:coint}
 Let $H$ be a Hopf $\group$-coalgebra, i.e.\ $H_1$ is a Hopf algebra. A \textit{left cointegral} of $H$ is an element $\coint\in H_1$, such that
 $$
  a \coint =\varepsilon(a)\coint\gc \qquad a\in H_1\gp
 $$
 Right cointegrals are defined similarly using the right multiplication with $a\in H_1$.
 \end{definition}
 
 It is known, see~\cite[Sec.\,10]{Rad}, that if $H_1$ is finite-dimensional left and right cointegrals always exist and are unique up to multiplication by scalars.
  
  \begin{definition}\label{def:grint}
  Let $H$ be a Hopf $\group$-coalgebra. \textit{A right $\group$-integral} of $H$ is a family of linear forms $\rint=(\rint_x)\in\prod_{x\in\group} H_x^*$ such that
  \be
  (\rint_x\otimes\id_{H_y})\circ\Delta_{x,y}(-)=\rint_{xy}(-)1_y\gp
  \ee
  A left $\group$-integral is defined similarly.
  \end{definition}
  It was proven in~\cite{HopfGC} that the space of left/right $\group$-integrals   is 1-dimensional  if $H$ is of finite type.
 
 The spaces for left and right (co)integrals  do not necessarily  coincide. In fact, if $\coint$ is a left cointegral, there exists a group-like functional $\modulus\in H_1^*$  such that  
\be\label{eq:modulus}
\coint a=\modulus(a)\coint
\ee
 for all $a\in H_1$. That is to say, $\modulus$ is an obstruction for a left cointegral to be a right cointegral.

Similarly, for a right $G$-integral $\rint$ there exists $\comodulus=(\comodulus_x)\in\prod_{x\in\group}H_x$ which is $\group$-group-like and such that the equality
$$
(\id_{H_x}\otimes\rint_y)\circ\Delta_{x,y}(-)=\rint_{xy}(-)\comodulus_x
$$
holds
 for all $x,y\in\group$. In other words, $\comodulus$ can  be understood as an obstruction for right $\group$-integrals to be left ones.
 
  \begin{definition}\label{def:distinguished}
  The element $\modulus\in H_1^*$ is called the $H_1$-distinguished group-like element or the \textit{modulus} of $H_1$.
  \\
  The $\group$-group-like element $\comodulus$ is called the $H$-distinguished $\group$-group-like element or the \textit{comodulus} of $H$. 
  \end{definition}
  \begin{remark}\label{rem:copint}
  A left $\group$-integral of a Hopf $\group$-coalgebra $H$ is a  right $\group^{\mathrm{op}}$-integral of the Hopf $\group^{\mathrm{op}}$-coalgebra $H^{\mathrm{cop}}$.
  We also note that the modulus of $H$ and $H^{\mathrm{cop}}$ coincide.
  \end{remark}
  
  \begin{definition}\label{def:unimodular}
  We call a Hopf algebra $H_1$ {\it unimodular} if $\modulus=\varepsilon$. A Hopf $\group$-coalgebra is called unimodular if $H_1$ is so.
  \end{definition}
  An example of the importance of these group-like elements can be seen in the following Lemma proven in ~\cite{HopfGC}.
\begin{lemma}\label{lem:rads4}
Let $H$ be a finite type Hopf $\group$-coalgebra, $\comodulus$ its comodulus and $\modulus$ its modulus. Then 
\be
 (S_{x^{-1}}S_x)^2(h)=\comodulus_x\left(\modulus^{-1}(h_{(1,1)})h_{(2,x)}\modulus(h_{(3,1)})\right)\comodulus_{x}^{-1}
\ee 
\end{lemma} 
\begin{remark}
In the Hopf algebra case, that is for the trivial group, this result recovers Radford's $S^4$ formula.
\end{remark}
 
 \subsection{$\modulus$-symmetrised $G$-integral} 
 We now recall the following statement proven in~\cite[Thm.\,4.2]{HopfGC}: 
 \begin{theorem}\label{prop:twistint}
  Let $H$ be a finite type Hopf $\group$-coalgebra. And $\rint=(\rint_x)_{x\in\group}$ a $\group$-integral then
  \be
   \rint_x(ab)=\rint_x\bigl(S_{x^{-1}}S_x(b\leftharpoonup \modulus)a\bigr)
  \ee
  where $\modulus$ is the modulus of $H_1$.
 \end{theorem}
  Assume that $(H,\pivot)$ is a pivotal Hopf $\group$-coalgebra. Define the \textit{$\modulus$-symmetrised}    $\group$-integral by shifting the argument of the integral by pivots:
  \be
  \hat{\rint}_x(a):=\rint_x(\pivot_xa)\gp
  \ee
   Similarly to integrals and due to invertibility of pivots, $\hat{\rint}_x$ is uniquely characterised by the relation
  \be\label{eq:Rintrel}
   (\hat{\rint}_x\otimes\pivot_y)\circ\Delta_{x,y}(a)=\hat{\rint}_{xy}(a)1_y\gc \qquad a\in H_{xy}\gp
  \ee
 
 Then from Theorem~\ref{prop:twistint} one can easily check the following.
\begin{proposition}\label{prop:twistgint}
Let $H$ be a Hopf $\group$-coalgebra and $\modulus$ its modulus. The $\modulus$-symmetrised $\group$-integral $\hat{\rint}_x$ factors through $\HH\bigl(H_x,R(\modulus)\bigr)$. Explicitly,
\be
\hat{\rint}_x(ab)=\hat{\rint}_x\bigl(R(\modulus)(b)a\bigr)\qcq \forall x\in\group,\; a,b\in H_x\gp
\ee
\qed
 \end{proposition}
 
 We note that in the unimodular case the $\modulus$-symmetrised integral $\hat{\rint}$ becomes a symmetric linear form called symmetrised integral~\cite{BBG}, and this explains our terminology.
 
\begin{notation}\label{not:trivialmod}
Let $H_1$ be a Hopf algebra and $V$ a vector space. We write ${}_\varepsilon V$ for the trivial $H_1$-module structure on $V$. That is, $h\cdot v=\varepsilon(h)v$ .
\end{notation}
We finish this section by recalling a generalisation of the result in~\cite[Thm.\,5.1]{BBG}  concerning the structure of tensor products of the regular representation to the case of $G$-coalgebras. 
 
  \begin{figure}
 \centering
\begin{subfigure}{0.5\textwidth} 
\centering
  \scriptsize {\def\svgwidth{2.5cm}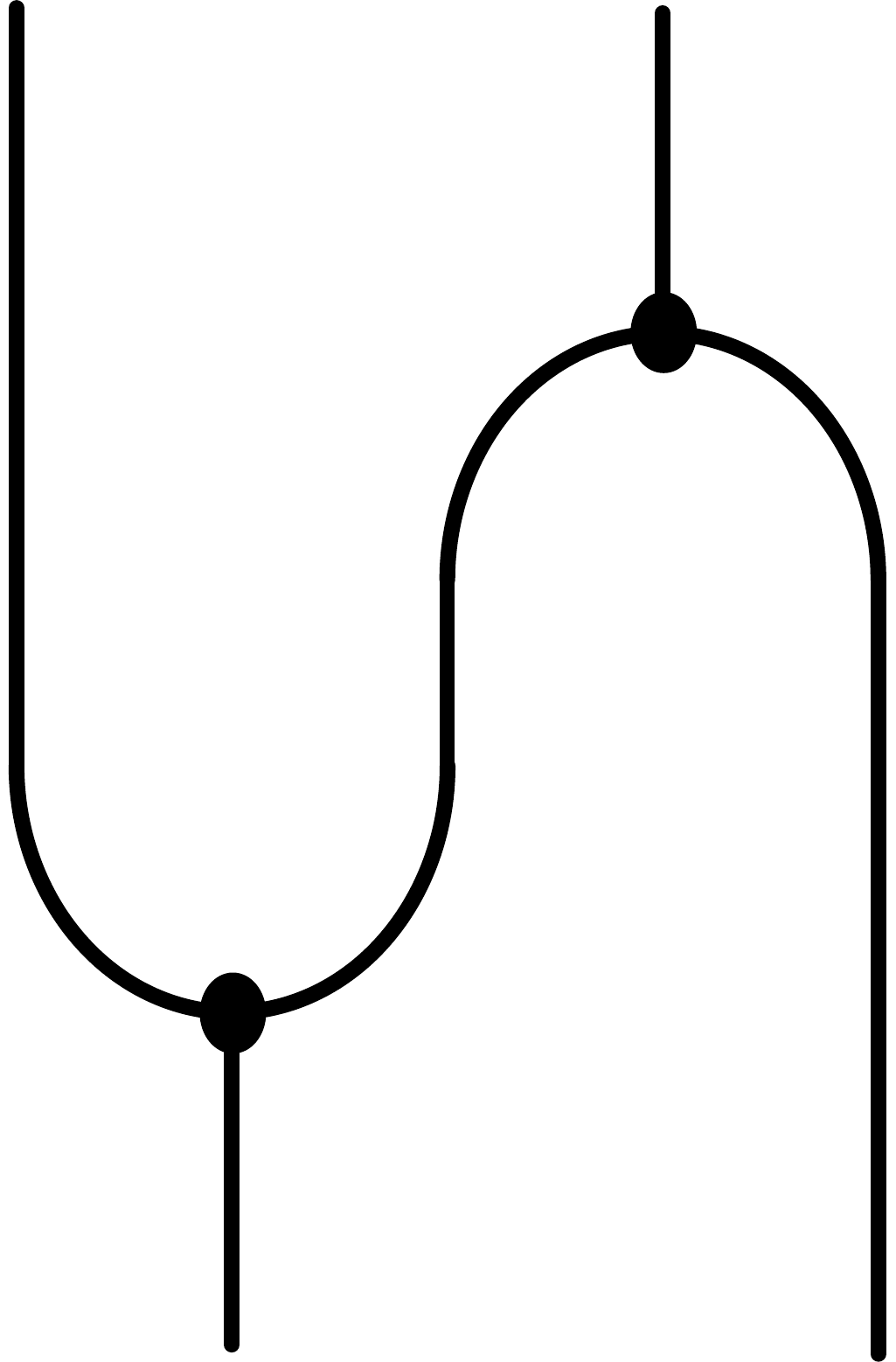}
  \caption{graphical representation of $\phi_{x,y}$}\label{fig:phi}
\end{subfigure}%
\begin{subfigure}{0.5\textwidth}
\centering
\scriptsize{\def\svgwidth{2.5cm}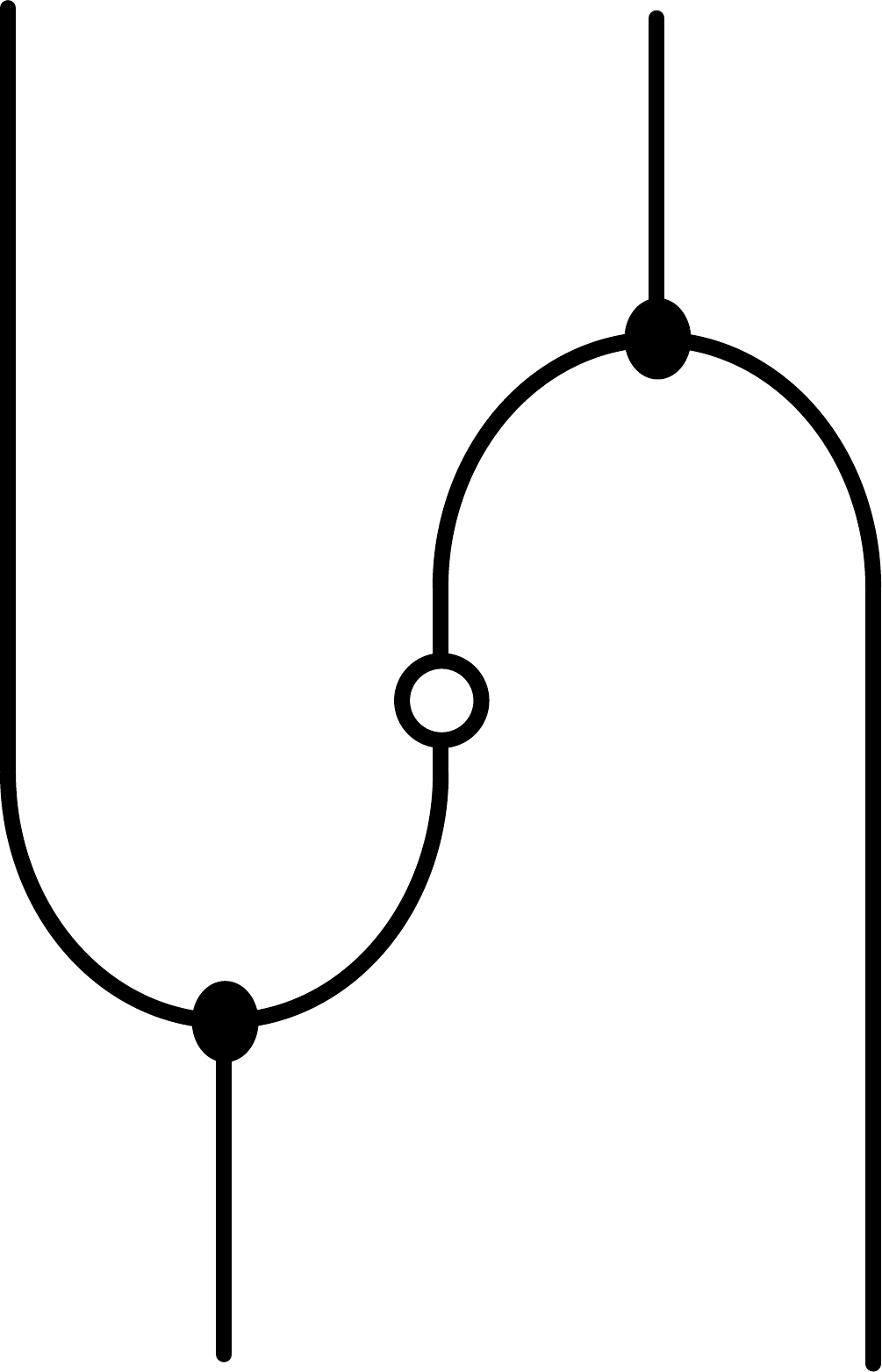}
  \caption{graphical representation of $\psi_{x,y}$}\label{fig:psi}
\end{subfigure} 
\caption{Maps from Theorem~\ref{thm:decoregular}}\label{fig:decoregular}
 \end{figure}
 
 \begin{theorem}[\cite{Phu}]\label{thm:decoregular}
Let $H=\{H_x\}$ be a Hopf $\group$-coalgebra.
 The maps 
 $$
  \phi_{x,y}\colon H_{xy}\otimes{}_\varepsilon H_y\rightarrow H_x\otimes H_y\qcq h\otimes m\mapsto h_{(1,x)}\otimes h_{(2,y)}m
 $$ 
 and
 $$
  \psi_{x,y}\colon H_x\otimes H_y\rightarrow H_{xy}\otimes {}_\varepsilon H_y\qcq a\otimes b\mapsto a_{(1,xy)}\otimes S_{y^{-1}}(a_{(2,y^{-1})})b 
 $$
 are $H_{xy}$-linear and inverse to each other.
 \end{theorem}

\section{Traces on $\Hpmod$}\label{sec:main}

In this section, we study module traces on the category $\Hpmod$ of projective modules over a Hopf $\group$-coalgebra $H$. First, we describe a certain type of module endofunctors on $\Hpmod$ we  use later for our main theorem. Then, we identify the space of trace maps on this particular type of endocategories with the $0$-th Hochschild homology of $H$ relative to an automorphism of $H$. Finally, we prove our main theorem that in particular states conditions when such trace maps satisfy the partial trace property \eqref{eq:rpartial}, and hence provides module traces, and we finally express them in terms of the $G$-integral.

\subsection{Trace maps for twist functors}
Recall that given an endomorphism $\varphi$ of an algebra~$A$, we can define the corresponding  twist functor $\varphi_*$  on the category of $A$-modules, see~\eqref{defeq:twistftor}. 

\begin{lemma}\label{lem:twistproj}
For $\varphi\in\Aut(A)$, the twist functor $\varphi_*$  restricts to an endofunctor on the projective ideal $\Apmod$.
\end{lemma}
\begin{proof}
Because $\varphi_*$ is an additive functor, it is enough to consider  indecomposable projective modules.
Assume that $V\in \Apmod$ is indecomposable, then there exists a projective $A$-module $M$ such that $V\oplus M$  is a  free $A$-module and it is isomorphic to $A$. Then we have
\be\label{eq:phi-additivie}
\varphi_*(A) \cong \varphi_*(V\oplus M) \cong \varphi_*(V)\oplus \varphi_*(M)\gp
\ee
We note that  the action on LHS $a\cdot_\varphi 1=\varphi(a)\cdot 1$ is zero if and only if $\varphi(a)$ is zero.
 But we assumed that $\varphi\in\Aut(A)$, i.e.\ $\varphi$ is an isomoprhism, therefore $a\cdot_\varphi 1$ is zero iff $a=0$. So, the module $\varphi_*(V\oplus M)$ is free as well, but this means that each direct summand of it is projective. Using~\eqref{eq:phi-additivie}, we thus conclude that $\varphi_*(V)$ is projective. By definition, $\Apmod$ is a full subcategory, and $\varphi_*$ is identity on morphisms, so we have that  $\varphi_*$ is an endofunctor on $\Apmod$.
\end{proof}

We now introduce the $0$-th Hochschild homology of $A$ relative to an automorphism $\varphi$:
$$
\HH(A,\varphi):=A\big/\bigl(ab-\varphi(b)a\bigr)\gp
$$
 Notice the similarity with~\eqref{endotra}.

\newcommand{\hphi}{\hat{\phi}}

\begin{theorem}\label{thm:hchextension}
 Let $A$ be a finite dimensional algebra and $\varphi\in\Aut(A)$. Then,
 \begin{align*}
 \phi\colon \HH(A,\varphi)&\rightarrow \HH(\Apmod,\varphi_*)\gc\\
 [a]&\mapsto [\varphi_*(r_a)\circ\varphi]
 \end{align*}
 is an isomorphism.
\end{theorem}
\begin{remark}
In~\cite{QLHo} the Hochschild complex is defined in its entirety and related to the corresponding Hochschild-Mitchel complex. 
\end{remark}
\begin{proof}
We first note  that the map $\varphi\colon A\rightarrow\varphi_*(A)$ is an isomorphism of $A$-modules. Therefore, we have an isomorphism
\begin{align}
\hphi\colon\; &A\xrightarrow{\;\cong\;}\Hom_A\bigl(A,\varphi_*(A)\bigr)\gc\notag\\
&a\mapsto \varphi_*(r_a)\circ\varphi\gc
\end{align}
 where $r_a$ denotes multiplication by $a\in A$ on the right.
 
 We see that $\varphi_*(r_{\varphi(a)})\circ\varphi=\varphi\circ r_a$. This implies that $\hphi$ descends to a map $\phi\colon \HH(A,\varphi)\rightarrow \HH(\Apmod,\varphi_*)$. Indeed,   
\begin{align*}
\hphi(ab-\varphi(b)a)&=\varphi_*(r_{ab-\varphi(b)a})\circ\varphi\\
&=\varphi_*(r_b)\circ\bigl(\varphi_*(r_a)\circ\varphi\bigr)-\varphi_*(r_a)\circ\bigl(\varphi_*(r_{\varphi(b)})\circ\varphi\bigr)\\
&=-\bigl(\hphi(a)\circ r_b-\varphi_*(r_b)\circ\hphi(a)\bigr)
\end{align*}
which is zero in $\HH(\Apmod,\varphi_*)$.
In order to check that $\phi$ is bijective, we construct its inverse as follows. Given an $A$-linear map $f\colon P\rightarrow\varphi_*(P)$ where $P$ is projective, take a decomposition of $\id_P=\sum_ia_i\circ\id_A\circ b_i$, with $a_i\colon A\rightarrow P,\,b_i\colon P\rightarrow A$, see~\cite[Lem.\,2.2]{BBG}, and define
\newcommand{\hpsi}{\hat{\psi}} 
\begin{align}
\hpsi\colon\,&\bigoplus\Hom_A\bigl(P,\varphi_*(P)\bigr) \longrightarrow A\notag\\
&f\mapsto\sum_i\varphi_*(b_i)\circ f\circ a_i(1)\gp
\end{align}

For $f\colon P'\rightarrow \varphi_*(P)$, and $g\colon P\rightarrow P'$, with $\id_P=\sum_ia_ib_i,\,\id_{P'}=\sum_ja_j'b_j'$, we have 
 \begin{align*}
 \hpsi(f\circ g)&=\sum_{i,j}\bigl(\varphi_*(b_i)\circ f\circ (a'_j\circ b'_j)\circ g\circ a_i\bigr)(1)\\
&=\sum_{i,j}\varphi\bigl((b_j'\circ g\circ a_i)(1)\bigr)\bigl(\varphi_*(b_i)\circ f\circ a'_j\bigr)(1)\\
&\sim\sum_{i,j}\bigl(\varphi_*(b_i)\circ f\circ a'_j\bigr)(1)(b_j'ga_i)(1)\\
&=\sum_{i,j}\varphi_*(b_j'\circ g\circ a_i)\bigl((\varphi_*(b_i)\circ f\circ a'_j\bigr)(1)1)\\
&=\sum_j\bigl(\varphi_*(b_j')\circ\varphi_*(g)\circ f\circ a_j'\bigr)(1)\\
&=\hpsi\bigl(\varphi_*(g)\circ f\bigr) \gp
 \end{align*}
We thus have a well-defined map $\psi\colon\HH(\Apmod,\varphi_*)\rightarrow \HH(A,\varphi)$. It can be seen that $\psi\circ\phi=\id_{\HH(A,\varphi)}$, by using the trivial decomposition of $\id_A$.

Let $f\colon P\rightarrow\varphi_*(P)$ and $\id_P=\sum_ia_ib_i$, define $a=\sum_i\bigl(\varphi_*(b_i)fa_i\bigr)(1)$. Then $r_a\circ\varphi=\sum_i\varphi_*(b_i)fa_i\sim f$, whence $\phi\circ\psi=\Id_{\HH(\Apmod,\varphi_*)}$.

Finally, to see that this construction does not depend on the choice of decomposition, we pick another decomposition $\id_P=\sum_ja'_j\circ \id_A\circ b_j'$  with $a'_j\colon A\rightarrow P,\,b'_j\colon P\rightarrow A$, then we have 
\begin{align*}
\hpsi(f)&=\sum_i\bigl(\varphi_*(b_i)\circ f\circ a_i\bigr)(1)\\
&=\sum_{i,j}\bigl(\varphi_*(b_i)\circ\varphi_*(a'_jb_j')\circ f \circ a_i\bigr)(1)\\
&=\sum_{i,j}\bigl(\varphi_*(b'_j)\circ f \circ a_i\bigr)(1)\varphi_*(b_ia'_j)(1)\\
&\sim\sum_{i,j}\varphi\bigl((b_ia_j')(1)\bigr)\bigl(\varphi_*(b'_j)\circ f \circ a_i\bigr)(1)\\
&=\sum_{i,j}\bigl(\varphi_*(b'_j)\circ f \circ a_i\circ b_i\circ a_j'\bigr)(1)\\
&=\sum_j\bigl(\varphi_*(b'_j)\circ f \circ a_j'\bigr)(1)\gc
\end{align*}
where `$\sim$' symbolises equality in $\HH(A,\varphi)$ and the equality after that is due to the $A$-linearity of $\varphi_*(b_j')\circ f\circ a_i\colon A\rightarrow\varphi_*(A) $.
\end{proof}

\begin{corollary}\label{cor:extTraces}
Let $A$ and $\varphi$ be as above. Then a $\varphi$-cyclic linear form $t$ on $A$ extends uniquely to a trace map $\{\t_P\colon\Hom_A\bigl(P,\varphi_*(P)\bigr)\}$, where
\be\label{eq:trExt}
 \t_P(f)=\sum_it\bigl(\left(\varphi_*(b_i)\circ f\circ a_i\right)(1)\bigr)
\ee
for a given decomposition $\id_P=\sum_ia_ib_i$. In particular,
\be
\t_A\bigl(\varphi_*(r_a)\circ\varphi\bigr)=t(a).
\ee
\end{corollary}
\begin{proof}
Let $t$ be a $\varphi$-cyclic form on $A$. It can be seen as a linear form on $\HH(A,\varphi)$. Since, $\HH(A,\varphi)\stackrel{\phi}{\cong} \HH(\Apmod,\varphi_*)$, we have that $t\circ\phi^{-1}$ provides a linear form on  $\HH(\Apmod,\varphi_*)$, and hence a trace map on $(\Apmod,\varphi_*)$, given by $t_P(f)=t\bigl(\phi^{-1}(f)\bigr)$ for all $P\in\Apmod, f\colon P\rightarrow\varphi_*(P)$. It can be easily checked that this agrees with \eqref{eq:trExt}.
\end{proof}

\begin{lemma}\label{lem:rpptrivial}
Let $(H,\pivot)$ be a finite type pivotal Hopf $\group$-coalgebra and $\Sigma$ be a module endofunctor on $\Hpmod$. Assume further that $\t$ is a trace map on $(\Hpmod,\Sigma)$. Then for all  $V\in\Hpmod$,  $W\in \vect$
and $f\colon V\otimes {}_\varepsilon W\rightarrow\Sigma(V\otimes {}_\varepsilon W)$ we have
\be\label{eq:rpptrivial}
  \t_{V\otimes {}_\varepsilon W}(f)=\t_V\bigl(\tr^\Sigma_{{}_\varepsilon W}(f)\bigr)\gp
\ee
\end{lemma}
\begin{proof}
Consider a decomposition of $\id_{{}_\varepsilon W }=\sum a_ib_i $ with $a_i\colon\ok\rightarrow {}_\varepsilon W,\, b_i\colon{}_\varepsilon W\rightarrow\ok $ with $\ev(b_i\otimes a_j)=b_i\bigl(a_j(1)\bigr)=\delta_{ij}$.

Let $f\colon V\otimes{}_\varepsilon W\rightarrow\Sigma(V\otimes{}_\varepsilon W)$, since the action on ${}_\varepsilon W$ is trivial we have $\ev_V(f\otimes v)=\widetilde{\ev}_V(v\otimes f)$ and hence 
\begin{align*}
\t_V\bigl(\tr^\Sigma_{{}_\varepsilon W}(f)\bigr)&=\sum_{i,j\in I}\t_V\left(\ipic{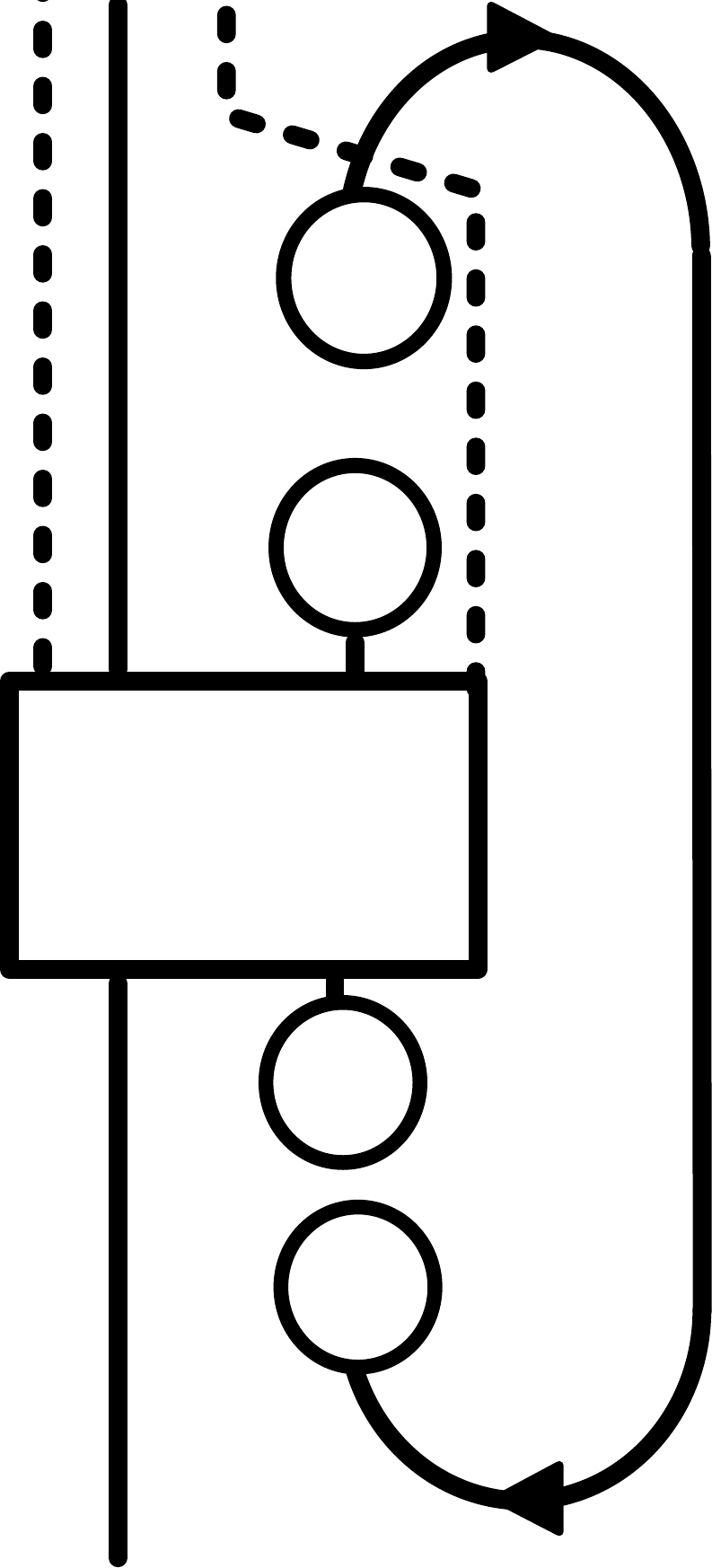}{0.2}\put(-26,31){\tiny$a_i $}\put(-26,13){\tiny$b_i $}\put(-27.5,-20.5){\tiny$a_j$}\put(-26.5,-34){\tiny$b_j$}\put(-35,-5){\scriptsize$f$} \right)\\
&=\sum_{i,j\in I}b_j(a_i)\t_V\bigl(\Sigma(\id_V\otimes b_i)\circ f\circ(\id_V\otimes a_j)\bigr)\\
&=\sum_{i\in I}\t_V\bigl(\Sigma(\id_V\otimes b_i)\circ f\circ(\id_V\otimes a_i)\bigr)\\
&=\sum_{i\in I}\t_{V\otimes {}_\varepsilon W}\bigl(f\circ(\id_V\otimes a_i)\circ(\id_V\otimes b_i)\bigr)\\
&=\t_{V\otimes {{}_\varepsilon W}}(f)\gp
\end{align*}
\end{proof}

\subsection{Main Theorem}
We are now ready to formulate and prove our main result.

\begin{theorem}\label{thm:main}
Let $(H,\pivot)$ be a finite type pivotal Hopf $\group$-coalgebra and $\modulus$ the modulus of~$H_1$. Then the space of module traces on $\bigl(\Hpmod,R(\modulus)_*\bigr)$ is equal to the space of $\modulus$-symmetrised right $\group$-integrals, and thus one-dimensional. Furthermore, the module trace $\t$ is non-degenerate and determined by
\be
\t_{H_x}(f)=\hat{\rint}_x\bigl(f(1_x)\bigr)\gc \qquad f\in\Hom_{H_x}\bigl(H_x,R(\modulus)_*(H_x)\bigr),\quad x\in G\gp
\ee
Similarly, the space of left module traces on $\bigl(\Hpmod,L(\modulus)_*\bigr)$ is determined by $\modulus$-symmetrised left $\group$-integrals via 
\be\label{eq:l-modtr-mu}
\t_{H_x}(f)=\hat{\rint}^l_x\bigl(f(1_x)\bigr)\gc \qquad f\in\Hom_{H_x}\bigl(H_x,L(\modulus)_*(H_x)\bigl),\quad x\in G\gp
\ee
\end{theorem}
\begin{proof}
First, we show that the $\modulus$-symmetrised $\group$-integral $\hat{\rint}$ provides a module trace. By Theorem \ref{thm:hchextension} this extends to a trace $\t$ on $\bigl(\Hpmod,R(\modulus)_*\bigr)$ with $\t_{H_x}(f)=\hat{\rint}_x\bigl(f(1_x)\bigr)$. According to Lemma~\ref{lem:gred}, $\t$ is a module trace only if \eqref{eq:gtGG} holds, where $\prog_x=H_x$. Using the maps in Figures~\ref{fig:phi} and~\ref{fig:psi} we get
\begin{align*}
\t_{H_x\otimes H_y}(f)&\stackrel{\text{Thm.} \ref{thm:decoregular}}{=}t_{H_x\otimes H_y}(f\circ\phi_{x,y}\circ\psi_{x,y})\\
&\stackrel{\eqref{eq:cyc}}{=}\t_{H_{xy}\otimes {}_\varepsilon H_y}\bigl(R(\modulus)_*(\psi_{x,y})\circ f\circ\phi_{x,y}\bigr)\\
&\stackrel{\text{Lem.} \ref{lem:rpptrivial}}{=}\t_{H_{xy}}\bigl(\tr_{{}_\varepsilon H_y}^\Sigma(R(\modulus)_*(\psi_{x,y})f\phi_{x,y})\bigr)\\
&=\ipic{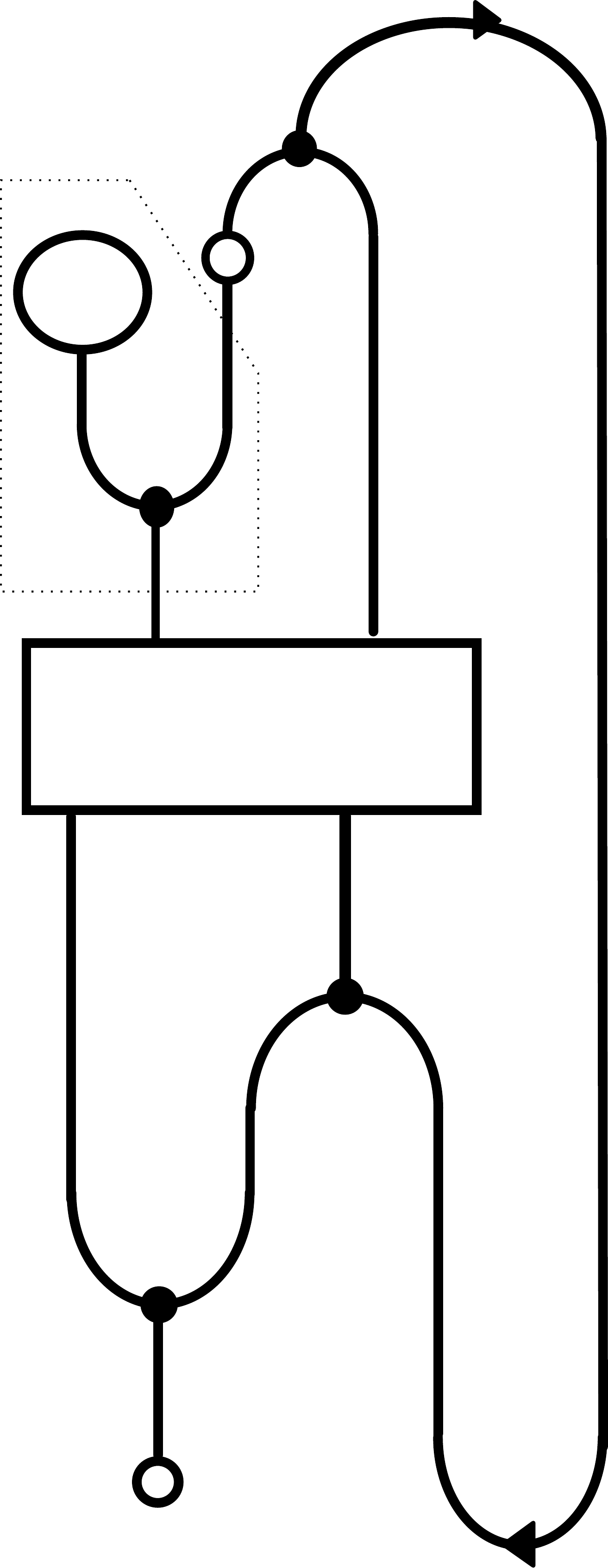}{.17}
\put(-40,0){$f$}
\put(-62,50){\scalebox{0.8}{\tiny$\hat{\rint}_{xy}$}}
\stackrel{*}{=}\ipic{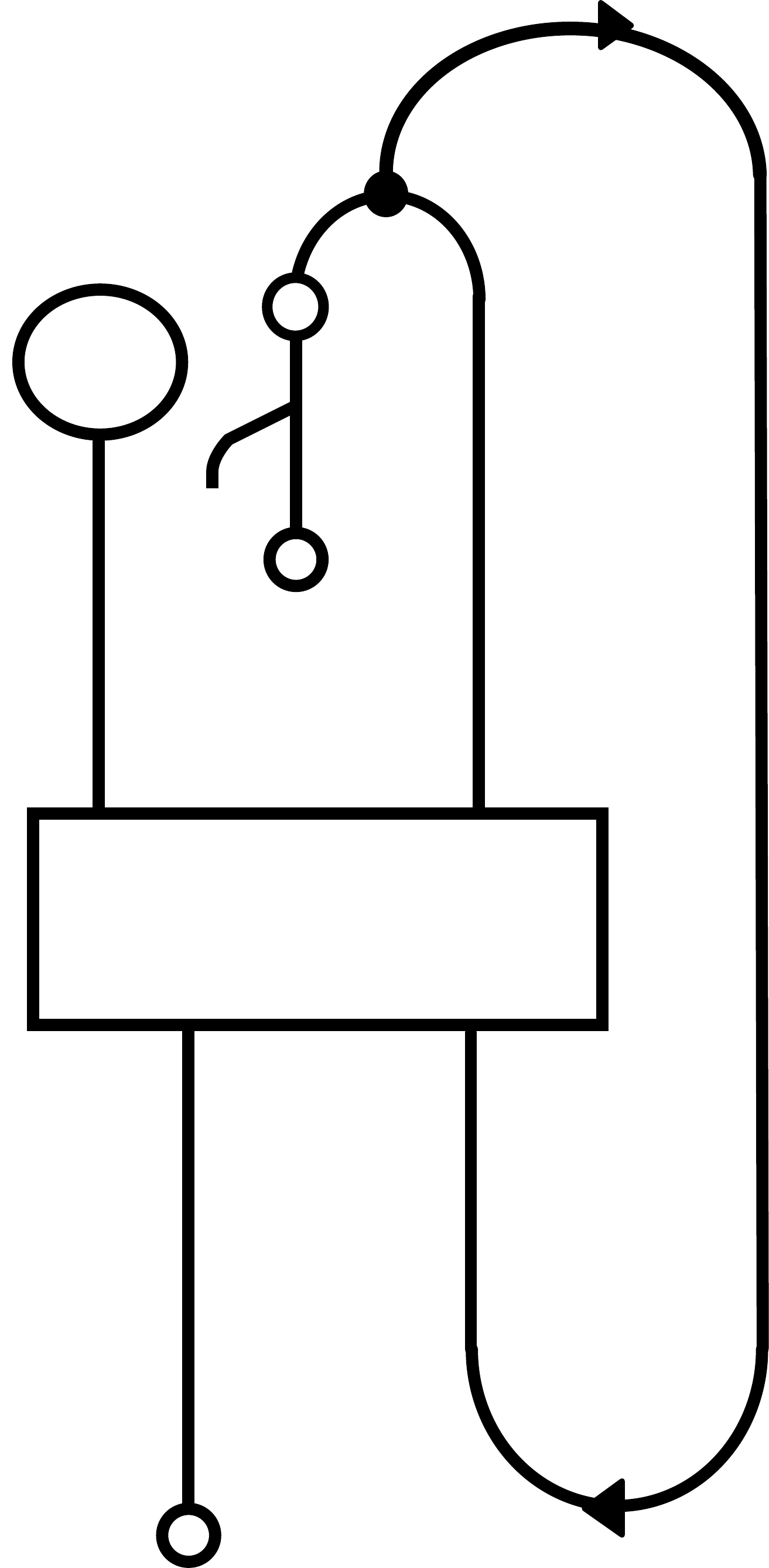}{.17}
\put(-40,-15){$f$}
\put(-62,33){\scalebox{0.8}{\tiny$\hat{\rint}_{x}$}}
\put(-55,18){\scalebox{0.75}{\tiny $\pivot_{y^{-1}}^{\text{-1}}$}}\\
&=\t_{H_x}\bigl(\tr^\modulus_{H_y}(f)\bigr)\gc
\end{align*} 
where in $(*)$ we used \eqref{eq:Rintrel} 
 to simplify the map in the dotted square.

Next assume there is a module trace $\t$ on $\bigl(\Hpmod,R(\modulus)\bigr)$. Define $\lambda_x(a):=\t_{H_x}\bigl(r_a\circ R(\modulus)\bigr)$. Then $\t_{H_x}(f)=\lambda_x\bigl(f(1_{H_x})\bigr)$. Notice that if $\beta\colon W\rightarrow\ok$ is a linear form, it can be seen as an $H_1$-linear map ${}_\varepsilon W\rightarrow\ok$. Then for all $x,y\in\group,\nu\in H_y^*,a\in H_{xy}$ we have
\begin{align*}
\lambda_{xy}(a)\nu(1_y)&=\t_{H_{xy}}\bigl((\varphi_{xy})_*(r_a)\circ R(\modulus)\nu(1_y)\bigr)
     =\t_{H_{xy}}\left(\ipic{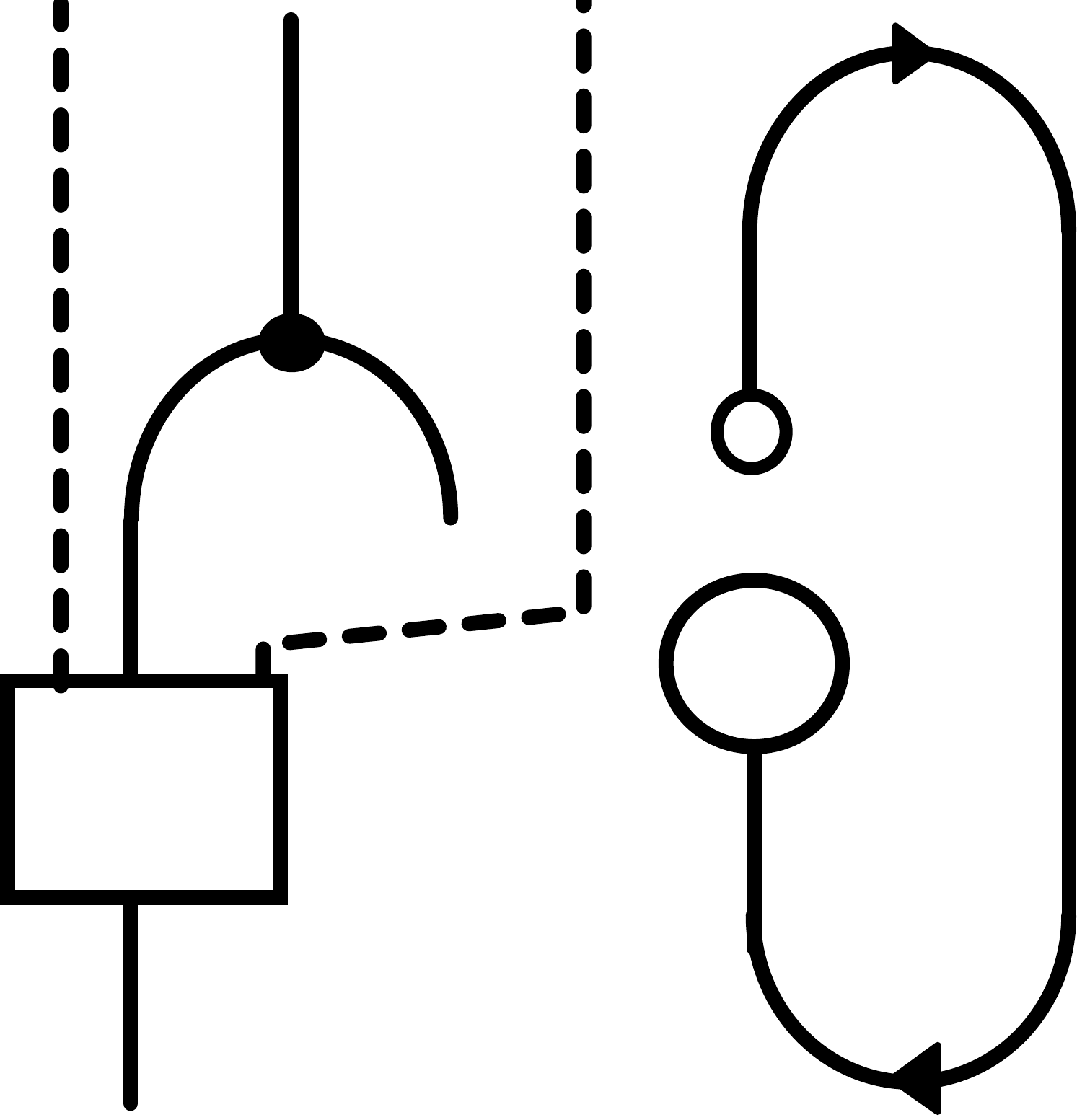}{.17}
     \put(-25,-10){\scriptsize$\nu$}\put(-73,-18){\scalebox{.9}{\scriptsize$R(\modulus)$}}\put(-50,0){\scriptsize$a$}\right)\\
     &\stackrel{\eqref{eq:rpartial}}{=}\t_{H_{xy}\otimes {}_\varepsilon H_y}\left(\ipic{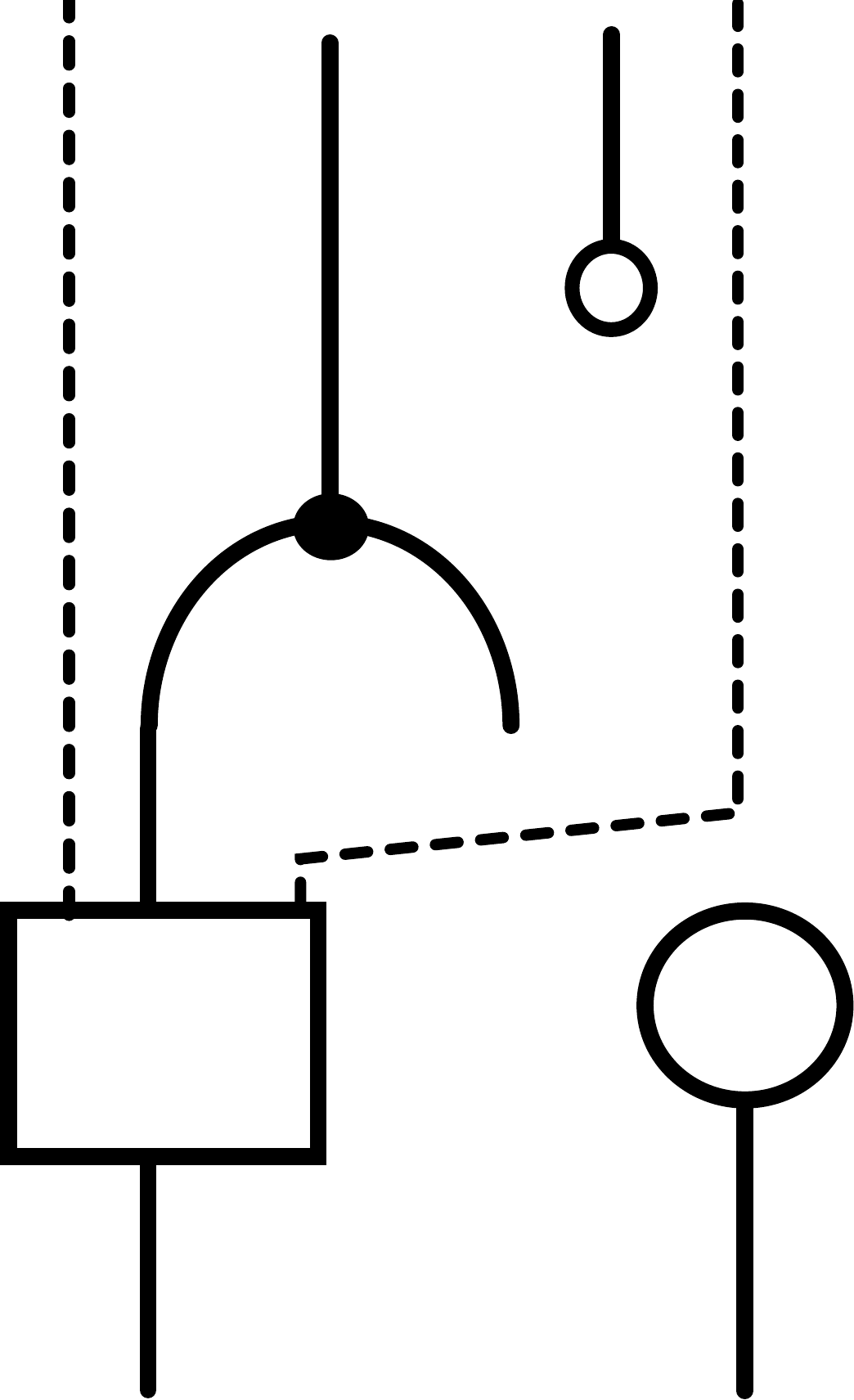}{.17}\put(-50,-22){\scalebox{0.9}{\scriptsize$R(\modulus)$}}\put(-27,-3){\scriptsize$a$}\put(-7,-20){\scriptsize$\nu$}\right)
     =\t_{H_{xy}\otimes {}_\varepsilon H_y}\left(\ipic{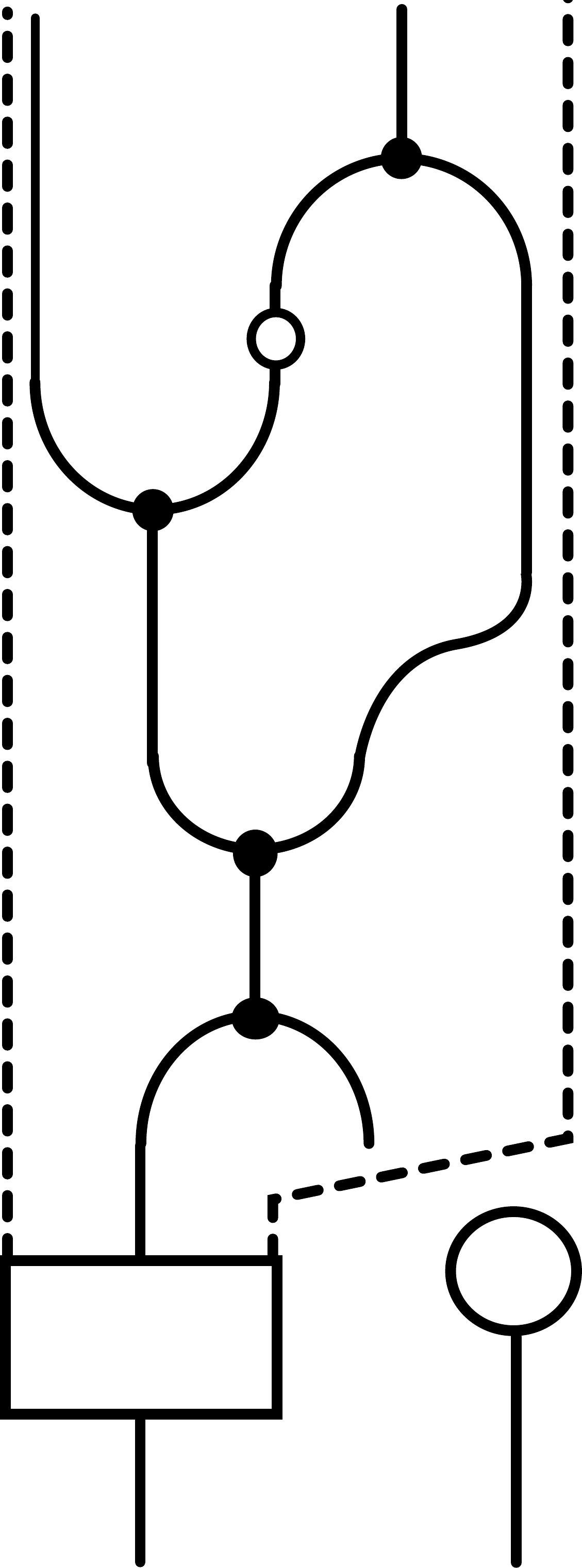}{.17}\put(-8,-50){\scriptsize$\nu$}\put(-51,-54){\scriptsize$R(\modulus)$}\put(-27,-35){\scriptsize$a$}\right)\\
     &\stackrel{\eqref{eq:cyc}}{=}\t_{H_x\otimes H_y}\left({\ipic{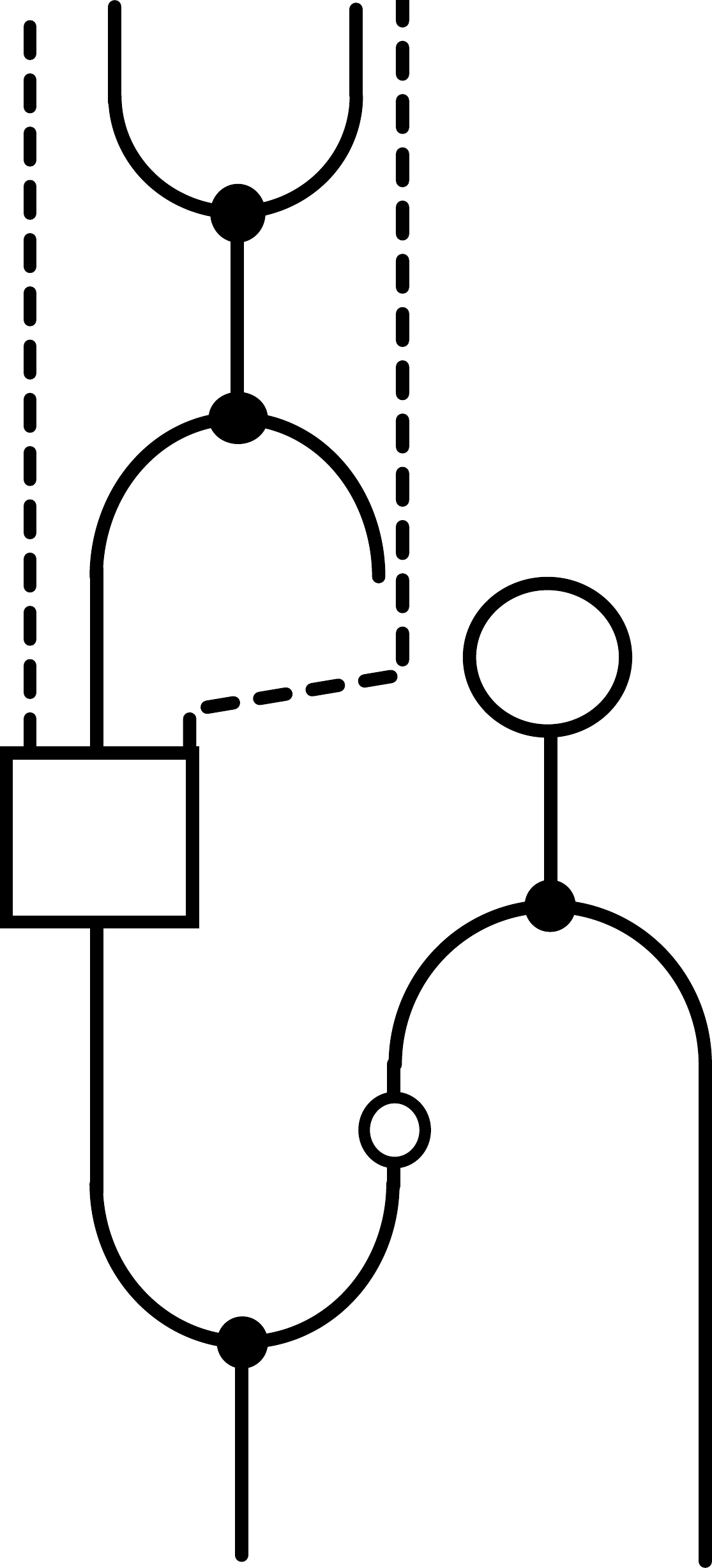}{0.17}}\put(-55,-5){\scalebox{0.8}{\tiny $R(\modulus)$}}\put(-16,8){\tiny $\nu$}\put(-35,10){\scriptsize $a$}\right)
     \stackrel{\eqref{eq:rpartial}}{=}\t_{H_x}\left(\ipic{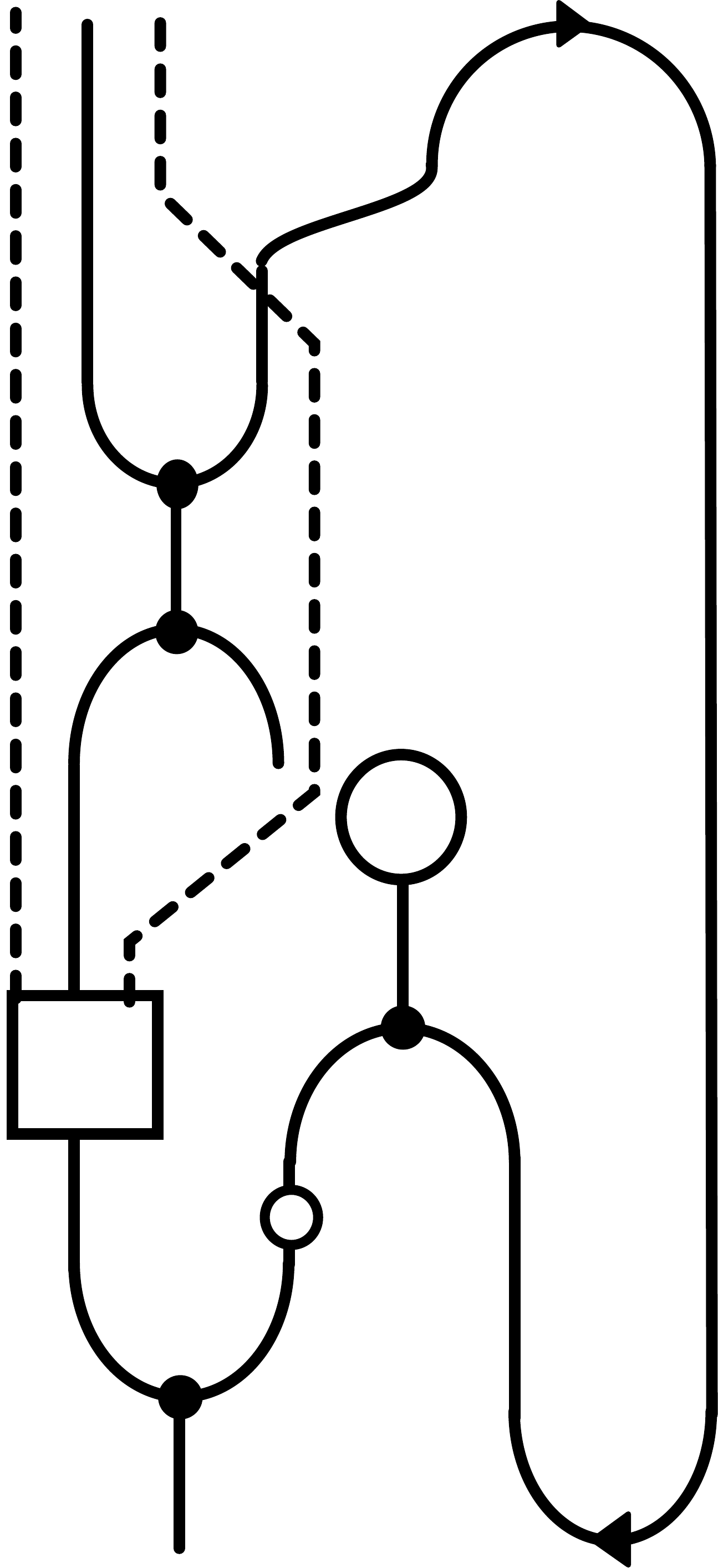}{0.17}\put(-63,-25){\scalebox{0.65}{\tiny $R(\modulus)$}}\put(-29,-5){\tiny $\nu$}\put(-45,0){\scriptsize $a$} \right)\\
    &=\ipic{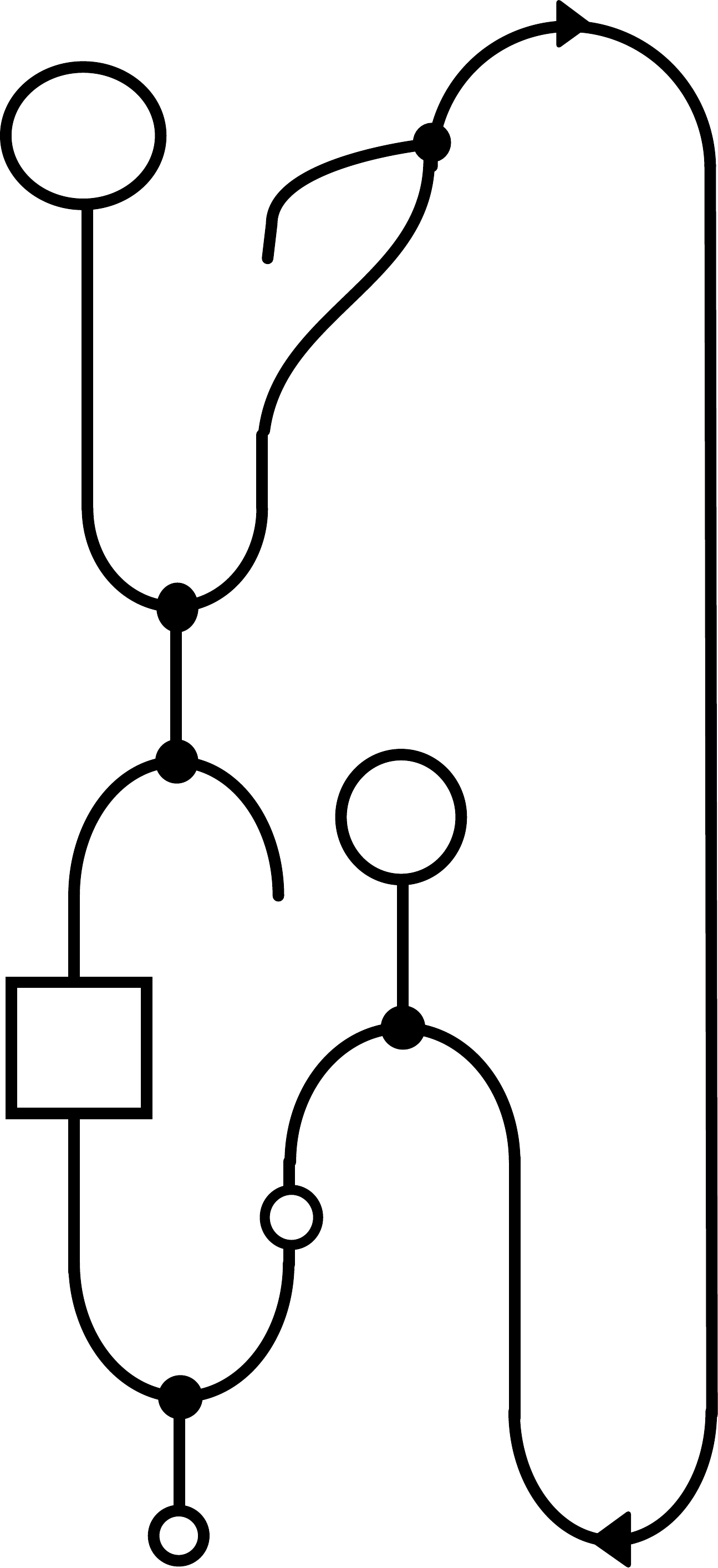}{0.17}\put(-63,-25){\scalebox{0.65}{\tiny $R(\modulus)$}}\put(-29,-5){\tiny $\nu$}\put(-45,-15){\scriptsize $a$}\put(-61,55){\tiny $\lambda_x$}\put(-50,50){\scriptsize$\pivot_y$}
     =\ipic{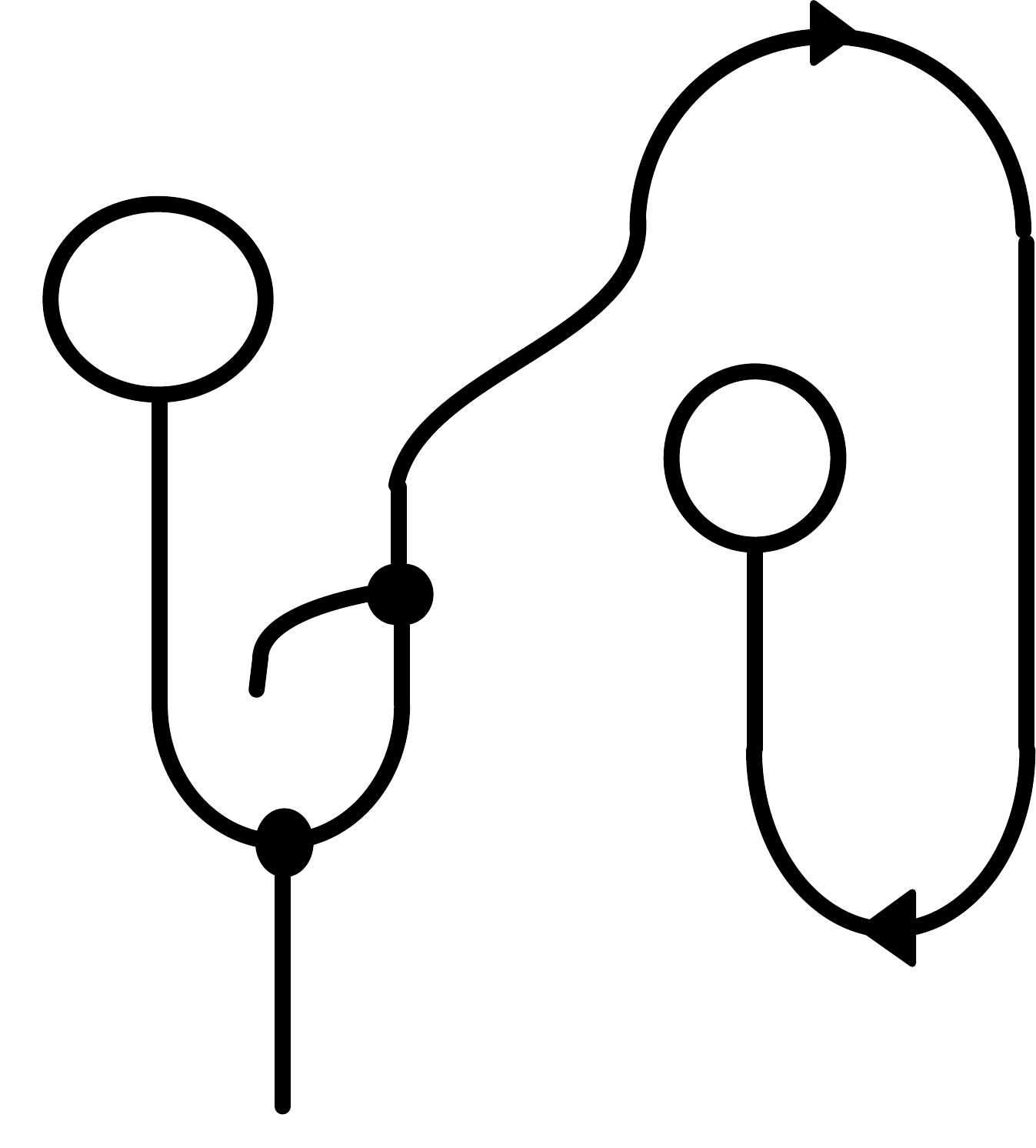}{0.17}\put(-20,7){\tiny $\nu$}\put(-50,-40){\scriptsize $a$}\put(-61,15){\tiny $\lambda_x$}\put(-50,3){\tiny$\pivot_y$}\\   
     &=\nu\bigl((\lambda_x\otimes\pivot_y)\Delta_{x,y}(a) \bigr)     \gp        
\end{align*}  
Since this relation holds for every $\nu\in H_y^*$ it follows that $\lambda_{xy}(a)1_y=(\lambda_x\otimes\pivot_y)\Delta_{x,y}(a)$. Hence $(\lambda_x)$ is a $\modulus$-symmetrised $\group$-integral.

We finally note that non-degeneracy of $\t$ follows immediately from that of $\hat{\rint}$.

 For the  case of left module traces, we first recall that a left $G$-integral of $H$ is a right $G$-integral of  the Hopf $G^{\mathrm{op}}$-coalgebra $H^{\mathrm{cop}}$ with the opposite coproduct. In this case, $S^{-1}_{x^{-1}}$ is the antipode whence $\pivot^{-1}$ works as the pivot.
 Then, the equality~\eqref{eq:l-modtr-mu} follows from applying the result for the right module trace to $H^{\mathrm{cop}}$. 
\end{proof}

\begin{corollary}
The endocategory $\bigl(\Hpmod,R(\nu)_*\bigr)$ has a non-zero right module trace if and only if $\nu=\modulus$ is the modulus of $H_1$.
Similarly, $\bigl(\Hpmod,L(\nu)_*\bigr)$ has a non-zero left module trace if and only if $\nu=\modulus$. In both cases, the module traces are non-degenerate.

In particular, $\Hpmod$ has a non-zero right (left) modified trace if and only if $H_1$ is unimodular. 
\end{corollary} 

\begin{proof}
Following the line of reasoning used in the second part of the previous proof implies that if $t$ is a module trace on $(\Hpmod,R(\gamma)_*)$
for some $\gamma\in H_1^*$ group-like,   it induces a family of maps $\{\lambda_x\colon H_x\rightarrow\ok\}$ which is both a $\modulus$-symmetrised $\group$-integral and satisfies $\lambda_x(ab)=\lambda_x((b\leftharpoonup\nu)a)$. Since the $\modulus$-symmetrised integral satisfies $\lambda_x(ab)=\lambda_x((b\leftharpoonup\modulus)a)$ by Proposition \ref{prop:twistint}, it must necessary be the case that $\lambda_x\bigl((b\leftharpoonup\modulus)a\bigr)=\lambda_x\bigl((b\leftharpoonup\nu)a\bigr)$. By the non-degeneracy of the $\group$-integral it follows that $b\leftharpoonup\modulus=b\leftharpoonup\nu$ whence $\nu=\modulus$.
\end{proof}

\section{Application to Hopf algebras}\label{sec:taft}
An important consequence of the theory of module traces for Hopf $G$-coalgebras formulated in Theorem~\ref{thm:main} occurs when $\group={1}$. In this case $H=H_1$ is a finite-dimensional pivotal Hopf algebra. Our main theorem implies then
\begin{corollary}\label{cor:Hopf}
Let $(H,\pivot)$ be a finite-dimensional pivotal Hopf algebra with modulus $\modulus\in H^*$. The space of right module traces on $\bigl(\Hpmod,R(\nu)_*\bigr)$ is $1$-dimensional for $\nu=\modulus$ and trivial otherwise. In the first case, the traces are non-degenerate and determined by
\be
 \t_H(f)=\rint\bigl(\pivot f(1)\bigr)
\ee 
for all $f\colon H\rightarrow R(\modulus)_*(H)$ and $\rint$ is a right integral.\\
Similarly, the space of left module traces  on $\bigl(\Hpmod,L(\nu)_*\bigr)$ is $1$-dimensional for $\nu=\modulus$ and trivial otherwise. And in the first case we have
\be
 \t_H(f)=\rint^l\bigl(\pivot^{-1} f(1)\bigr)
\ee 
for all $f\colon H\rightarrow L(\modulus)_*(H)$  and $\rint^l$ is a left integral.\\
In particular, $\Hpmod$ has a non-zero right or left modified trace in the sense of the definition in~\eqref{def:mod-tr}  if and only if $H$ is unimodular. 

\end{corollary}

Many examples for the unimodular case were described in~\cite[Sec.\,7\&8]{BBG}, where the modified trace was explicitly calculated for the small quantum groups of ADE type. 
In the remainder of this section, we show examples of computation of module traces for non-unimodular Hopf algebras. An important class of such Hopf algebras are the Taft algebras.

%\subsection*{Taft Hopf algebras}
\vskip-3mm
\mbox{}\\
{\bf Taft Hopf algebras.}\;
Fix $r\in\oN$, and $\zeta$ a primitive $r$-th root of $1$. Consider the $\oC$-algebra $T_r$ generated by $E,K$ with the relations
$$
 E^r=0\qcq K^r=1\qcq KE=\zeta EK\gp
$$
It has the basis $\{E^iK^j\}_{i,j=0,1,...,r-1}$, i.e.\ an algebra of dimension $r^2$.
The Hopf algebra structure is
$$
\begin{array}{lll}
 \Delta(E)=1\otimes E+E\otimes K\gc\quad &\varepsilon(E)=0\gc &\quad S(E)=-EK^{-1}\gc\\
 \Delta(K)=K\otimes K\gc &\varepsilon(K)=1\gc &\quad S(K)=K^{-1}\gp
\end{array}
$$
We see then that $S^2(E)=KEK^{-1}$ and therefore $K$ provides a pivot element. The coproduct can be written on the basis as 
\be\label{eq:copbas}
 \Delta(E^iK^j)=\sum_{l=1}^{i}\left[\begin{array}{c}i\\ l \end{array}\right]_\zeta E^lK^j\otimes E^{i-l}K^{l+j}\gp
\ee 
One can check that the linear form 
\be\label{eq:Taftrint}
\rint(E^mK^n)=\zeta\delta_{m,r-1}\delta_{n,1}
\ee
 is a right integral for the algebra, which gives us the $\modulus$-symmetrised integral:
 \be\label{eq:Taft-hatmu}
 \hat{\rint}(E^mK^n)=\delta_{m,r-1}\delta_{n,0}\gp
 \ee
 We have also   a left cointegral: 
  \be
  \coint=E^{r-1}\sum_{i=0}^{r-1}\zeta^{-i}K^i\gp
  \ee 
   The modulus of $T_r$  can easily be checked to be given by 
   \be\label{eq:Taftmodulus}
   \modulus(K)=\zeta \gc \qquad \modulus(E)=0\gp
   \ee

For all $s=0,1,...,r-1$, the elements 
\be\label{eq:Taft-es}
e_s=\frac{1}{r}\sum_{i=0}^{r-1}\zeta^{si}K^i
\ee
are orthogonal primitive idempotents and $\sum_s e_s=1$. Hence the projective indecomposable  modules are given by 
\be\label{eq:Ps-def}
P_s:=T_r\cdot e_s\gp
\ee
and these are projective covers of the irreducible representations (which are all 1-dimensional) where $E$ acts by zero and $K$ acts by $\zeta^{-s}$.  Observe that  
\be\label{eq:K-es}
Ke_s=\zeta^{-s}e_s
\ee 
and so $P_s$ has the linear basis 

\be
P_s = \langle \, E^ie_s\, | \,  i=0,1,...,r-1\,\rangle_\ok\gp
\ee 

From the definition~\eqref{eq:Ps-def}, we see that all intertwining maps $f\colon P_s\rightarrow P_l$ are given by right multiplication with an element of the form $e_s\cdot a\cdot e_l$, for $a\in T_r$. The space of the intertwiners is one dimensional for all pairs $(s,l)$ since 
\be\label{eq:space-int}
\Hom_{T_r}(P_s,P_l) \cong
e_s\cdot T_r\cdot e_l=\left\{\begin{array}{lr}
\left\langle  E^{l-s}e_l\right\rangle_\ok \gc \quad & l\geq s\gc\\
\left\langle   E^{r+l-s}e_{l}\right\rangle_\ok \gc \quad& s> l\gc
\end{array} \right.
\ee
where we used~\eqref{eq:K-es} and the commutation relations $E e_l = e_{l-1} E$  
together with the orthogonality condition $e_s e_l = \delta_{l,s} e_l$.

We further study the automorphism $R(\modulus)$ and the corresponding endofunctor $R(\modulus)_*$, introduced in ~\eqref{eq:rhook} and~\eqref{defeq:twistftor}. 
Since $R(\modulus)$ is an algebra morphism it suffices to compute $R(\modulus)(E)=E$
and $R(\modulus)(K)=\zeta K$.

We notice then that the $\ok$-linear isomorphism  
\be\label{eq:twistexample}
\begin{array}{lccl}
f_s\colon &R(\modulus)_*(P_{s})&\xrightarrow{\; \cong \;} &P_{s-1}\\
& E^ie_{s} &\mapsto & E^ie_{s-1}
\end{array}
\ee
is also $T_r$-linear, here we identified $P_{-1}\equiv P_{r-1}$ and $e_{-1}\equiv e_{r-1}$ . Indeed,
\begin{align*}
f_s(K\cdot_{R(\modulus)}E^ie_{s})&=\zeta\zeta^i\zeta^{-s}E^ie_{s-1}=\zeta^iE^iKe_{s-1}=Kf_s(E^ie_{s})\gc\\
f_s(E\cdot_{R(\modulus)}E^ie_s)&=E^{i+1}e_{s-1}=E^if_s(e_s)\gp
\end{align*}

Now we are ready to compute a module trace $\t$ on $\bigl(T_r\text{-pmod},R(\modulus)_*\bigr)$, which is a family of maps
$$
\t_{P_s}\colon \; \Hom_{T_r}\bigl(P_s,R(\modulus)_*(P_s)\cong P_{s-1}\bigr) \to \mathbb{C}.
$$
 Recall that this trace map exists and is expressed via the integral following Corollary~\ref{cor:Hopf}.

We begin with the case $s>0$. By~\eqref{eq:space-int} for $l=s-1$, the space  $\Hom_{T_r}\bigl(P_s,R(\modulus)_*(P_s)\bigr)$ is spanned by the map $R^s=f_s^{-1}\circ r_{E^{r-1}e_{s-1}}$:
\begin{align}\label{eq:Psnot0gen}
R^s\colon \; &P_s\rightarrow P_{s-1}\xrightarrow{\;\cong\;} R(\modulus)_*(P_s)\notag\\
&e_s\mapsto E^{r-1}e_{s-1}\mapsto E^{r-1}e_s
\end{align}
and the  module trace on it is computed to be 
$$
\t_{P_s}(R^s)=\hat{\rint}(R^s(e_s))=\hat{\rint}(E^{r-1}e_s)=\frac{1}{r}\gc
$$
where we used~\eqref{eq:Taft-hatmu}.

For $s=0$, the space
  $\Hom_{T_r}\bigl(P_0,R(\modulus)_*(P_0)\bigr)$
 is spanned by $R^0=f_0^{-1}\circ r_{E^{r-1}e_{r-1}}$
\begin{align}\label{eq:P0r-morph}
R^0\colon \;& P_0\rightarrow P_{r-1}\xrightarrow{\;\cong\;} R(\modulus)_*(P_0)\gp\notag\\
&e_0\mapsto E^{r-1}e_{r-1}\mapsto E^{r-1}e_0\gp
\end{align}
  Then the module trace on the morphism~\eqref{eq:P0r-morph} can be computed to be
$$
\t_{P_0}(R^0)=\hat{\rint}\left(R^0(e_0)\right)=\hat{\rint}\left(E^{r-1}e_{0}\right)=\frac{1}{r}\gp
$$
In both cases, we used Corollary~\ref{cor:extTraces} with the decomposition 
  $$
  T_r\xrightarrow{\; r_{ e_s}\;} T_re_s\hookrightarrow T_r\gc
  $$
  i.e. we chose $a_i$ as the inclusion and $b_i=r_{e_s}$ .

\section{$G$-graded case: non-restricted quantum $sl(2)$}\label{sec:Borel}
In this section, we give an example of a module trace for Hopf group-coalgebras, and it is based on Borel quantum groups.

Recall that \textit{non-restricted} quantum algebra $\Uqua$ is a Hopf algebra generated by $E,F,K,K^{-1}$ with relations 
$$
KE=q^2EK\qcq KF=q^{-2}FK\qcq[E,F]=\frac{K-K^{-1}}{q-q^{-1}}\qcq KK^{-1}=1=K^{-1}K
$$
where $q=e^{\frac{i\pi}{r}}$  is a $2r$-th root of $1$. It has a Hopf algebra structure given by
$$
\begin{array}{lcl}
\Delta(E)=1\otimes E+E\otimes K\gc\quad&\varepsilon(E)=0\gc\quad& S(E)=-EK^{-1}\gc\\
\Delta(F)=K^{-1}\otimes F+F\otimes 1\gc\quad&\varepsilon(F)=0 \gc\quad&S(F)=-KF\gc\\
\Delta(K)=K\otimes K \gc\quad&\varepsilon(K)=1\gc\quad& S(K)=K^{-1}\gp
\end{array}
$$
It is worth noting that $E^r,F^r,K^r$ are central elements of this algebra and that $\pivot=K^{1+nr}$ provides a pivotal structure for any $n\in\mathbb{Z}$. Consider the Hopf subalgebra $C'=\C[F^r,E^r,K^{\pm r}]$, and $\group'=Alg(C',\C)$.

\begin{proposition}\label{prop:charqgp}
$\group'\cong\C\rtimes\C^\times\ltimes\C$, where the left action of $\C^\times$ on $\C$  is given by $z\rhd x=zx$ and the right action $y\lhd z=z^{-1}y$
\end{proposition} 
\begin{proof}
Define the map $\group'\ni f\mapsto \left(f(F^r),f(K^r),f(E^r)\right)$. This is clearly a bijection. To see that it is a morphism of groups for $f,g\in\group$  compute $fg(K^r)=f(K^r)g(K^r)$ and 
\begin{align*}
fg(E^r)&=\sum_{i=0}^r\left[\begin{array}{c}r\\i\end{array}\right]_{q^2}f(E^i)g(E^{r-i}K^i)\\
&=g(E^r)+f(E^r)g(K^r)\\
\intertext{and similarly,}
fg(F^r)&=f(F^r)+f(K^r)^{-1}g(F^r)\gp
\end{align*}
\end{proof}

Using Example \ref{ex:quotHGC} (2), we can define a finite type Hopf $\group'$-coalgebra $\{(\Uqua)_{y,z,x}\}_{(y,z,x)\in \group'}$
 where $(\Uqua)_{y,z,x}$ is $\Uqua$ modulo the relations $E^r=x,\, F^r=y,\, K^r=z$, with $(y,z,x)\in\group'$ under the identification from Proposition \ref{prop:charqgp}. In \cite{Phu}, the Hopf group-coalgebra coming from the subgroup $\{0\}\times\C^\times\times\{0\}\subseteq\group'$ was studied. 

%\subsection{Borel subalgebra and its trace}
\vskip-3mm
\mbox{}\\
{\bf Borel subalgebra and its trace.}\;
We are interested in the subalgebra $B$ generated by $E,K$, also called the positive Borel subalgebra.
 Let $C=\C[E^r,K^{\pm r}]$ and $\group=Alg(C,\C)\cong \C^\times\ltimes\C$. We have a finite type Hopf $\group$-coalgebra $\{B_{z,x}\}_{(z,x)\in \group}$. Here, each $B_{z,x}$ is the algebra $B$ modulo the extra relations $K^r=z$ and $E^r=x$. It inherits the pivotal structure given by $B_{z,x}\ni\pivot_{z,x}=K^{nr+1}=z^nK$ .
 
We note that each $B_{z,x}$ for non-zero $x$ is isomorphic to the matrix algebra $\mathrm{Mat}_{r\times r}(\oC)$, while $B_{1,0}$ is the Taft algebra from the previous section where one has to set $\zeta=q^2$. For the latter algebra we thus immediately get the integral  $\rint_{(1,0)}(E^iK^j)=\delta_{i,r-1}\delta_{j,0}$ as in~\eqref{eq:Taftrint}. 
Using Definition~\ref{def:grint} we can compute in general
\begin{align*}
\rint_{(z,x)}(E^lK^j)1_{(z,x)}&=\sum_{i=0}^l\Biggl[\begin{array}{c}l\\i\end{array}\Biggr]_{q^2}\rint_{(1,0)}(E^iK^j)E^{l-i}_{(z,x)}K^{j+i}_{(z,x)}\\
&=\left\{\begin{array}{lc}
0 & l\neq r-1\gc\\
\rint_{(1,0)}(E^lK^j)K^r_{(z,x)}& l=r-1\gp
\end{array}\right.
\end{align*}
That is, 
\be
\rint_{(z,x)}(E^iK^j)=z\delta_{i,r-1}\delta_{j,1}
\ee
 is a right $\group$-integral, and hence 
 \be
 \hat{\rint}_{(z,x)}(E^iK^j)=z^{1+n}\delta_{i,r-1}\delta_{j,0}
 \ee
  is a $\modulus$-symmetrised $\group$-integral.

The algebra $B_{z,x}$ has primitive idempotents
$$
e_s=\frac{1}{r}\sum_{i=0}^{r-1}z^{-i/r}q^{2is}K^i
$$
and thus there are $r$ projective indecomposables $V_s=B_{z,x}e_s$ with bases $\{v_i^s=E^ie_s\}$ with the action given by $Kv_i^s=z^{-1/r}q^{2(i-s)}v_i^s$, $Ev_i^s=v_{i+1}^s$ for $i\neq r-1$ and $Ev_{r-1}^s=xv_0.$

 We notice that they can be also considered as representations  of the original (infinite-dimensional) Borel algebra $B$. Indeed, $B$  covers each $B_{z,x}$, so by the pull-back the modules above are $B$-modules for all $(z,x)$. Moreover, it is easy to see that all indecomposable projectives in the category of finite-dimensional $B$-modules are classified by these ones.

Using ~\eqref{eq:Taftmodulus} we get $\modulus(E)=0$, $\modulus(K)=q^2K$ and therefore $R(\modulus)(E)=E$, $R(\modulus)(K)=q^2K $ for all $B_{z,x}$. We then get $R(\modulus)_*(V_s)\cong V_{s-1} $ via the map $v^s_{i}\mapsto v^{s-1}_i$. 
Notice that
$$
\Hom_{B_{z,x}}(V_s,V_l)\cong e_sB_{z,x}e_l=\left\{\begin{array}{lr}
\langle E^{l-s}e_l\rangle_\ok\gc& l\geq s\gc\\
\langle E^{r+l-s}e_l\rangle_\ok\gc & s>l\gc
\end{array}\right.
$$
is one dimensional.
In particular, $\Hom_{B_{z,x}}(V_s,R(\modulus)_*(V_s))$ is generated by $R_{z,x}^s\colon e_s\mapsto E^{r-1}e_s$. Using Corollary~\ref{cor:extTraces} we get a right module trace determined by $$\t_{V_s}(R_{z,x}^s)=\hat{\rint}_{(z,x)}(E^{r-1}e_s)=\frac{z^{1+n}}{r}\gp $$

\medskip
\appendix
\section{Module Categories}\label{app:modcat}
Let $\CC$ be a monoidal category. A (right) module category over $\CC$ is a tuple $(\modcat,\otact,m,r)$ where $\modcat$ is a category and $\otact:\modcat\times\CC\rightarrow\modcat$ a functor, and ${m:\otact(\otact\times\Id)\rightarrow\otact(\Id\times\otimes)}$, ${r:(-)\otact\mathbb{1}\rightarrow\Id_\modcat}$  natural isomorphisms such that the following diagrams commute for all $M\in\modcat,\,C,D,E\in\CC$
\begin{equation}\label{diag:penta}
\xymatrix@R=20pt@C=40pt@W=10pt@M=10pt{
((M\otact C)\otact D)\otact E\ar[r]^{m}\ar@<-0pt>[dd]_{m\otact\id}&
(M\otact C)\otact(D\otimes E)\ar@<-0pt>[d]_m\\
& M\otact(C\otimes(D\otimes E))\\
 (M\otact(C\otimes D))\otact E  \ar[r]^{m}& M\otact((C\otimes D)\otimes E) \ar[u]_{\id\otact \alpha}
}
\end{equation}
\be\label{diag:triang}
 \xymatrix{
 (M\otact\mathbb{1})\otact C\ar[rr]^m\ar[ddr]^{r\otact\id} & & M\otact(\mathbb{1}\otimes C)\ar[ldd]_{\id\otact\lambda}\\
 & &\\
  &M\otimes C &
 }
\ee
where $\alpha$ is the associativity constraint of $\CC$ and $\lambda$ the left unitor.

Given $\modcat$, $\modcat'$ two module categories over $\CC$, \textit{a module functor} from $\modcat $ to $\modcat'$ is a pair $(\Sigma,\sigma)$ where $\Sigma:\modcat\rightarrow\modcat'$ is a functor and the family of natural isomorphisms $\sigma_{M,C}\colon\Sigma(M\otact C)\rightarrow\Sigma(M)\otact C$ is such that the following diagrams commute:
\be
\xymatrix{
\Sigma\left((M\otact C)\otact D\right)\ar[r]^{\Sigma m}\ar[dd]_\sigma &\Sigma \left(M\otact(C\otimes D)\right)\ar[d]_\sigma \\
& \Sigma (M)\otact (C\otimes D)\\
\Sigma(M\otact C)\otact D\ar[r]^{\sigma\otact\id}&\left(\Sigma(M)\otact C\right)\otact D\ar[u]_{m}
}
\ee
\be
 \xymatrix{
 \Sigma(M\otact\mathbb{1})\ar[rd]^{\Sigma r}\ar[rr]^{\sigma}& &\Sigma(M)\otact\mathbb{1}\ar[ld]_\rho\\
 &\Sigma (M)
 } 
\ee

It is important to note that often $\CC$ carries an extra structure or property, and in this case it is common to assume that $\modcat$ also does, and further to require that $\otimes$ is compatible with this extra structure in some sense. For instance, in case $\CC$ is a $\ok$-linear  category $\modcat$ is usually also considered to  be a $\ok$-linear  category, and $\otimes$ is required to be $\ok$-bilinear.

\begin{remark}
We notice that while $\Hmod$ is $\ok$-linear \textsl{abelian}
our main example of module categories over it -- $\Hpmod$ --  is not abelian. However, it is $\ok$-linear and the action $\otact:=\otimes$ is $\ok$-bilinear. Furthermore, we assume endofunctors of $\Hpmod$ to be $\ok$-linear as well. 
\end{remark} 

%\bibliographystyle{alpha}
%\bibliography{TwistSymm}
\newcommand\arxiv[2]      {\href{http://arXiv.org/abs/#1}{#2}}
\newcommand\doi[2]        {\href{http://dx.doi.org/#1}{#2}}
\newcommand\httpurl[2]    {\href{http://#1}{#2}}

\end{document}